\documentclass[11 pt, usletter]{article}

\usepackage[latin1]{inputenc}
\usepackage{amsmath,amsfonts,hyperref,amsthm, enumitem, graphicx, comment}
\usepackage[textsize=footnotesize,color=green!40]{todonotes}
\usepackage{fullpage}
\usepackage{authblk}

\newtheorem{thmalpha}{Theorem}

\newtheorem{lemmaalpha}[thmalpha]{Lemma}

\newtheorem{thm}{Theorem}
\newtheorem{lemma}[thm]{Lemma}
\newtheorem{prop}[thm]{Proposition}
\newtheorem{cor}[thm]{Corollary}
\theoremstyle{definition}
\newtheorem*{definition*}{Definition}
\newtheorem{rem}{Remark}
\newtheorem*{cla}{Claim}

\newcommand{\Aut}{\text{Aut}}
\newcommand{\Good}{\text{Good}}

\newcommand{\cA}{\mathcal{A}}
\newcommand{\cB}{\mathcal{B}}
\newcommand{\cE}{\mathcal{E}}
\newcommand{\cF}{\mathcal{F}}
\newcommand{\cG}{\mathcal{G}}
\newcommand{\cH}{\mathcal{H}}
\newcommand{\cJ}{\mathcal{J}}

\newcommand{\cT}{\mathcal{T}}
\newcommand{\cU}{\mathcal{U}}
\newcommand{\cV}{\mathcal{V}}
\renewcommand{\Pr}{\mathbf{Pr}}

\newcommand{\sm}{\setminus}

\newcommand{\removedStuff}[1]%
{{\leavevmode\color{blue}{{\tiny 
here is some stuff I removed or moved elsewhere:
#1
}}}}

\renewcommand{\removedStuff}[1]{}

    \makeatletter
    \let\@fnsymbol\@arabic
    \makeatother

\title{Local convergence and stability \\of tight bridge-addable graph classes}

\author[$\clubsuit$]{Guillaume Chapuy%
\thanks{%
Support from \emph{Agence Nationale de la Recherche}, grant number ANR~12-JS02-001-01 ``Cartaplus'', and from the City of Paris, grant ``\'Emergences Paris 2013, Combinatoire \`a Paris''.
Email:~{\tt guillaume.chapuy@liafa.univ-paris-diderot.fr}.
}
}
\author[$\spadesuit$]{Guillem Perarnau%
\thanks{
Email:~{\tt g.perarnau@bham.ac.uk}.
}
}
\affil[$\clubsuit$]{\small\it {\sc IRIF, UMR CNRS 8243},
Universit\'e Paris-Diderot,
France.}
\affil[$\clubsuit$]{\small\it {\sc CRM, UMI CNRS 3457},
Universit\'e de Montr\'eal,
Canada.
}
\affil[$\spadesuit$]{\small\it School of Mathematics, University of Birmingham, Birmingham, United Kingdom.
}

\begin{document}

\maketitle

\vspace{-10mm}
\begin{abstract}
A class of graphs is \emph{bridge-addable} if given a graph $G$ in the class, any graph obtained by adding an edge between two connected components of $G$ is also in the class. 
The authors recently proved a conjecture of McDiarmid, Steger, and Welsh stating that if $\mathcal{G}$ is bridge-addable and $G_n$ is a uniform $n$-vertex graph from $\mathcal{G}$, then $G_n$ is connected with probability at least $(1+o_n(1))e^{-1/2}$. The constant $e^{-1/2}$ is best possible since it is reached for the class of all forests.

In this paper we prove a form of uniqueness in this statement: if $\mathcal{G}$ is a bridge-addable class and the random graph $G_n$ is connected with probability close to  $e^{-1/2}$, then $G_n$ is asymptotically close to a uniform $n$-vertex random forest in some local sense.
For example, if the probability converges to $e^{-1/2}$, then $G_n$ converges in the sense of Benjamini-Schramm to 
the uniform infinite random forest $F_\infty$.
This result is reminiscent of so-called ``stability results'' in extremal graph theory, with the difference that here the stable extremum is not a graph but a graph class.

\end{abstract}

\section{Introduction and Main Results}

In this paper graphs are simple. A graph is labeled if its vertex set is of the form $[1..n]$ for some $n\geq 1$. An unlabeled graph  is an equivalence class of labeled graphs by relabeling. Unless mentioned otherwise, graphs are labeled.
A class of (labeled) graphs $\mathcal{G}$ is \emph{bridge-addable} if given a graph $G$ in the class, and an edge $e$ of $G$ whose endpoints belong to two distinct connected components, then $G\cup\{e\}$ is also in the class.
Examples of bridge-addable classes include planar graphs, forests, or $H$-free graphs where $H$ is any 2-edge connected graph (see many more examples in~\cite{ABMCR,CP15}).
 
McDiarmid, Steger and Welsh~\cite{MCSW} conjectured that every bridge-addable class of graphs with $n$ vertices contains at least a proportion $(1+o_{n}(1))e^{-1/2}$ of connected graphs. This has recently been proved by the authors.
In the next statement and later, we denote by $\cG_n$ the set of graphs in $\cG$ with $n$ vertices, and by $G_n$ a uniform random element of $\cG_n$.
\begin{thmalpha}\label{thm:conj}[Chapuy, Perarnau~\cite{CP15}]
For every $\epsilon>0$, there exists an $n_0$ such that for every bridge-addable class $\cG$  and every $n\geq n_0$,  we have
\begin{align}\label{eq:conj}
\Pr \left( G_n \mbox{ is connected} \right) \geq (1-\epsilon)e^{-1/2}.
\end{align}
\end{thmalpha}
\noindent If $\cG$ is the class of all forests (which is bridge-addable), then Theorem~\ref{thm:conj} is asymptotically tight, since it is shown in~\cite{Renyi} that if $F_n$ is a uniform random forest on $n$ vertices, then as $n$ goes to infinity:
\begin{eqnarray}\label{eq:Renyi}
\Pr\left(F_n \mbox{ is connected}\right) \longrightarrow e^{-1/2}.
\end{eqnarray}
The aim of this paper is to show that, in some appropriate sense, this class is
locally the only one
 for which Theorem~\ref{thm:conj} is asymptotically tight. More precisely, we will show that any bridge-addable class of graphs that comes close to achieving the constant $e^{-1/2}$ is ``close'' to a uniform random forest in some local 
sense.

\begin{definition*}
For any $\zeta>0$, we say that a bridge-addable class of graphs $\cG$ is \emph{$\zeta$-tight with respect to connectivity} (or simply \emph{$\zeta$-tight})  
if there exists an $n_0$ such that for every $n\geq n_0$ we have
\begin{align*}
\Pr \left( G_n \mbox{ is connected} \right)\leq (1+\zeta)e^{-1/2}\;,
\end{align*}
where we recall that $G_n$ is chosen uniformly at random from $\cG_n$.
\end{definition*}

In order to state our results, we first need to introduce some notation and terminology. If $H$ is a graph we let $|H|$ be its number of vertices.
We denote by $\cU$ the set of \emph{unlabeled, unrooted} trees 
 and by $\cT$ the set of \emph{unlabeled, rooted} trees, \textit{i.e.} trees with a marked vertex called the root.  For every unrooted tree $U\in \cU$, we denote by $\Aut_u(U)$ the number of automorphisms of $U$, and for every rooted tree $T\in \cT$, we denote by $\Aut_r(T)$ the number of automorphisms of $T$ that fix its root. 
Moreover given $k$ unrooted trees $U_1,\dots, U_k$ in $\cU$, we denote by $\Aut_u(U_1,\dots,U_k)$ the number of automorphisms of the forest formed by disjoint copies of $U_1,\dots,U_k$.

\newcommand{\sma}{\mathop{Small}}
Given a graph $H$, we let $\sma(H)$ denote the graph formed by all the components of $H$ that are not the largest one (in case of a tie, we say that the largest component of the graph is the one with the largest vertex label among
 all candidates). 
In what follows, we will always see $\sma(H)$ as an unlabeled graph. 
Given a graph $G$ and a rooted tree $T\in \cT$, we let $\alpha^G(T)$ be the number of pendant copies of the tree $T$ in $G$. More precisely, $\alpha^G(T)$ is the number of vertices $v$ of $G$ having the following property: there is at least one cut-edge $e$ incident to $v$, and if we remove the such cut-edge that separates $v$ from the largest possible component, the vertex $v$ lies in a component of the graph that is a tree, rooted at $v$, which is isomorphic to $T$.
The following is classical:
\begin{thmalpha}[{see Appendix~\ref{app:gf}}]\label{thmalpha:randomForests}
Let $F_n$ be a uniform random forest with $n$ vertices. Then for any fixed unlabeled unrooted forest ${\mathbf{f}}$ we have as $n$ goes to infinity:
\begin{eqnarray}\label{eq:convSmallForest}
\Pr\big(\sma(F_n) \equiv {\bf f}\big)
\longrightarrow p_\infty ({\bf f}) :=e^{-1/2}\frac{e^{-|{\bf f}|}}{\Aut_u({\bf f})}, 
\end{eqnarray}
where $\equiv$ denotes unlabeled graph isomorphism.
Moreover, $p_\infty$ is a probability distribution on the set of unlabeled unrooted forests.

For any fixed rooted tree $T\in \cT$ we have as $n$ goes to infinity:
\begin{eqnarray}\label{eq:convAlphaForest}
\frac{\alpha^{F_n}(T)}{n} \stackrel{(p)}{\longrightarrow} a_\infty(T)
:=\frac{e^{-|T|}}{\Aut_r(T)}.
\end{eqnarray}
 where (p) indicates convergence in probability.
Moreover $a_\infty$ is a probability measure on $\cT$.
\end{thmalpha}

\medskip
Our main result says that, for a bridge-addable class $\cG$, if we have an approximate version of~\eqref{eq:Renyi} for $G_n$, then we also have an approximate version of \eqref{eq:convSmallForest} and~\eqref{eq:convAlphaForest} for $G_n$. 
In the next statement and everywhere in the paper, the constants $\epsilon, \eta, \rho, \nu, \zeta$ are implicitly assumed (in addition to other written quantifications or assumptions) to be positive and smaller than $c$ where $c$ is a small, absolute, constant.
\begin{thm}[\bf Main result]\label{thm:main}
For every $\epsilon, \eta>0$,
there exist a $\zeta>0$ 
and an $n_0$
 such that for every $\zeta$-tight bridge-addable class $\mathcal{G}$ and every $n\geq n_0$, the following holds:
\begin{enumerate}
\item[i)]
The small components of $G_n$ are close  to those of a large random forest, in the sense that for every unrooted unlabeled forest $\mathbf{f}$
we have:
$$
\Big|\Pr\big(\sma(G_n)\equiv {\bf f}\big)
-
p_\infty({\bf f})
\Big|
<\epsilon.
$$
\item[ii)]
The statistics of pendant trees in $G_n$ are close to those of a large random forest, in the sense that:
$$
\Pr\left(\forall T\in \cT: 
\left|\frac{\alpha^{G_n}(T)}{n}-a_\infty(T)\right|<\eta
\right)
>1-\epsilon.
$$
\end{enumerate}
\end{thm}
\begin{rem}
It is easy to see (up to adapting the dependence of $\zeta$ in $\epsilon, \eta$)  that we can replace $i)$ by:
\begin{enumerate}[topsep=1pt]\it
\item[i')]
The total variation distance between the law of $\sma(G_n)$ and the probability law $p_\infty$ is at most $\epsilon$.
\end{enumerate}
Similarly we could replace $ii)$ by:
\begin{enumerate}[topsep=1pt]\it
\item[ii')]
The $L^1$-distance between the measure $\alpha^{G_n}(\cdot)/n$ and the probability law $a_\infty(\cdot)$ (both are measures on $\mathcal{T}$) is at most $\eta$ with probability at least $1-\epsilon$.
\end{enumerate}
\end{rem}

\begin{rem}
Our main result, Theorem~\ref{thm:main}, can be viewed both as a \emph{unicity result} (since it states that in the limit, and through the lens of local observables, the class of forests is the only one to reach the optimum value $e^{-1/2}$) and as a \emph{stability result} (since it also states that the only classes than come \emph{close} to the extremal value $e^{-1/2}$ are \emph{close} to forests, again through local observables of random graphs).
Here we use the terminology ``stability result'' on purpose, by analogy with the field of extremal graph theory. 
Indeed the study of graphs that come \emph{close} to achieving extremal properties is a classical topic in this field.
\emph{Stability results}, pioneered in the 
papers~\cite{erdos1966limit,erdos1967some,erdos1966some,simonovits1968method}, show that in many cases, the graphs that are close to being extremal have a structure close to the actual extremal graphs, in some quantifiable sense.
 Our main result suggests that the question of stability of extremal graph
classes, with respect to appropriate graph limit topologies (here, local convergence), should be
further examined.
\end{rem}

Before going into the proof of the theorem, let us look at some closely related statements and corollaries. 
Call a bridge-addable class $\cG$ \emph{tight} if it is $\zeta$-tight for any $\zeta>0$, that is to say, as $n$ goes to infinity:
$$
\Pr(G_n \mbox{ is connected}) \rightarrow e^{-1/2}.
$$
Then we obtain the following consequence of Theorem~\ref{thm:main} for tight bridge-addable classes, which is weaker (it is a unicity, but not a stability result).
\begin{thm}[\bf Convergence of local statistics in tight bridge-addable graph classes]
\label{thm:tight}
Let $\mathcal{G}$ be a tight bridge-addable class of graphs.
Then, when $n$ goes to infinity, $\sma(G_n)$ converges in distribution (for the discrete topology) to the probability measure $p_\infty(.)$ given by~\eqref{eq:convSmallForest}. Equivalently, for all unrooted forests {\bf f}, we have:
\begin{align}\label{eq:thmtightPart1}
\Pr(\sma(G_n)\equiv {\bf f}) \longrightarrow p_\infty({\bf f}).
\end{align}
Moreover, for any rooted tree $T\in \cT$, the proportion $\frac{\alpha^{G_n}(T)}{n}$ of local pendant copies of the tree $T$ converges in probability to the deterministic constant $a_\infty(T)$ given by~\eqref{eq:convAlphaForest}:
\begin{align}\label{eq:thmtightPart2}
\frac{\alpha^{G_n}(T)}{n} \stackrel{(p)}{\longrightarrow} a_\infty(T).
\end{align}
In particular, the random measure $\frac{1}{n}\alpha^{G_n}(\cdot)$ on $\mathcal{T}$ converges in distribution to the deterministic probability measure $a_\infty(\cdot)$ (for the topology on measures induced by convergence in distribution on discrete spaces).
\end{thm}

Theorem \ref{thm:tight} states that, from the point of view of statistics of pendant trees and of non-largest components, tight classes are  indistinguishable from random forests in the limit. 
Let us develop in this direction. Let $V_n$ be a uniform random vertex in $G_n$. 
Then for a given $T\in \cT$, conditionally to $G_n$, the quantity $\alpha^{G_n}(T)/n$ is the probability that from $V_n$ hangs a copy of the tree $T$. Readers familiar with the Benjamini-Schramm (BS) convergence of graphs~\cite{BenjaminiSchramm}  will note the similarity with this notion.
If $(G,x)$ and $(H,y)$ are two rooted graphs (here and in the forthcoming discussion we allow infinite graphs but we always implicitly assume  them to be locally finite), define the BS-distance $d_{BS}((G,x);(H,y))$ as $2^{-K}$ where $K$ is the largest integer such that the balls of radius $K$ in $(G,x)$ and $(H,y)$ are isomorphic, as rooted unlabelled graphs.
This distance
 (also called the \emph{ball distance}, see~\cite{Lovasz:book}) 
defines a topology on the set of rooted graphs, and enables us to talk about convergence in distribution of random rooted graphs, in the BS-sense. An equivalent definition of this convergence is the following: a sequence of random rooted graphs $(H_n,x_n)$ converges to $(H_\infty, x_\infty)$ if and only if for any $r\geq 1$ and for any rooted graph $(H,x)$ of radius $r$, the probability that the ball of radius $r$ in $(H_n,x_n)$ is isomorphic to $(H,x)$ converges to the probability of the same event in $(H_\infty, x_\infty)$. 

It is classical (see e.g.~\cite[Lemma 2.4]{Aldous})
that if $F_n$ is a uniform random forest on $n$ vertices rooted at a random uniform vertex $V_n$, then
$$
(F_n,V_n) \rightarrow (F_\infty, V_\infty)
$$
in distribution in the BS-sense, where $(F_\infty,V_\infty)$ is the ``uniform infinite rooted random forest'' (which we could also have called ``uniform infinite rooted random tree'', since it is almost surely a tree). Namely, $(F_\infty,V_\infty)$ can be constructed as follows: consider a semi-infinite path, starting at a vertex $V_\infty$, and identify each vertex of this path with the root of an independent Galton-Watson tree with offspring distribution $Poisson(1)$.
In our context, passing from pendant trees to balls is an easy task, and one can deduce the following from Theorem~\ref{thm:tight}.
\begin{cor}[\bf Local convergence of tight bridge-addable graph classes]\label{cor:BS}
Let $\mathcal{G}$ be a tight bridge-addable graph class. Let $G_n$ be a uniform random graph from $\cG_n$ and let $V_n$ be a uniform random vertex of $G_n$. Then $(G_n,V_n)$ converges to $(F_\infty, V_\infty)$ in distribution in the Benjamini-Schramm sense.
\end{cor}
The purpose of stating Corollary~\ref{cor:BS} is to illustrate the link between our local observables and the BS topology, but we could have stated stronger intermediate results. For example Corollary~\ref{cor:BS}  only uses the second part of Theorem~\ref{thm:tight},
and says nothing about the small connected components. In fact, it follows from Theorem~\ref{thm:tight} that for tight bridge-addable classes,  the pair $((G_n,V_n), \sma(G_n))$ converges in distribution to $(F_\infty,V_\infty)\otimes p_\infty$ for the product of the BS and the discrete topologies.  

Also note 
that the last corollary
is of a weaker nature than Theorem~\ref{thm:tight} since it averages over the random graph $G_n$, while~\eqref{eq:thmtightPart2} shows that a single sample of the random graph $G_n$ is locally similar to $F_\infty$ with high probability. More precisely,
it is possible to formulate another, stronger corollary of Theorem~\ref{thm:tight} in terms of the BS convergence as follows. Let $\mu_{G_n}$ be the law, given $G_n$, of the random rooted graph $(G_n,V_n)$ where $V_n$ is a uniform vertex of $G_n$ (thus $\mu_{G_n}$ is a random probability measure on the set of locally finite rooted graphs). Then it follows from Theorem~\ref{thm:tight} 
that if $\cG$
is a tight
bridge-addable class, $\mu_{G_n}$ converges in distribution to the \emph{deterministic} probability measure $\mu_\infty$, defined as the law of $(F_\infty, V_\infty)$. 
The underlying topology for the convergence is the topology of weak convergence induced by the BS distance on the set of probability measures on rooted graphs. 
We will not give more details on these questions since they only concern weaker reformulations of Theorem~\ref{thm:tight}.

\medskip
\begin{rem}\label{rem:example}
Our main theorem asserts that $\zeta$-tight bridge-addable classes are ``locally similar'' to random forests in some precise sense. However, they can be very different from some other perspective. 
For example, consider the set $\tilde {\mathcal{F}_n}$ of graphs on $[1..n]$ defined as follows: $\tilde {\mathcal{F}_n}$ contains the graph in which all edges linking vertices in $[1..\lceil n^{2/3}\rceil]$ are present and all other vertices are isolated, and $\tilde {\mathcal{F}_n}$ is the smallest bridge-addable class containing this graph. In other words, $\tilde{\mathcal{F}_n}$ is the set of graphs inducing a clique on $[1..\lceil n^{2/3}\rceil]$, and such that contracting this clique gives a forest.
Then $\tilde{\mathcal{F}}=\cup_{n\geq 1}\tilde{\mathcal{F}_n}$ is a bridge-addable class, and it is easy to see that it is tight (see Appendix~\ref{subapp:example} for more details), so our main results apply. However one can argue that the random graph $\tilde F_n$ in $\tilde{\mathcal{F}}_n$ is \textit{very} different from a random forest in several senses: first, it has $\Theta(n^{4/3})$ edges whereas a forest has linearly many. Second, with probability $1-O(n^{-1/3})$ an edge taken uniformly at random from $\tilde F_n$ belongs to a clique of size $\lceil n^{2/3}\rceil$, which is very different from what happens in a forest. This last point does not contradict our results, but only recalls that it is important here to think of locality as a measure of what happens around ``typical vertices'' and not ``typical edges''.
\end{rem}

\begin{rem}\label{rem:example2}
While our results state that typical graphs in tight bridge-addable classes look like random forests in a ``local'' sense, they do not imply that this similarity is preserved in a ``global'' sense. 
For example one can  ask whether a random graph in a tight bridge-addable class, properly normalized, converges as a metric space to the Continuous Random Tree (CRT) with respect to the Gromov-Hausdorff (GH) topology. This is true for uniform random forests~\cite{Aldous2}, but a simple example shows that this is not the case in general. 
Following the lines of the example presented in Remark~\ref{rem:example}, consider the set $\hat {\mathcal{F}_n}$ of graphs on $[1..n]$ defined as follows: $\hat {\mathcal{F}_n}$ contains the graph  where the vertices $1,\dots,\lceil n^{2/3}\rceil$ induce a path and all the other ones are isolated, and $\hat {\mathcal{F}_n}$ is the smallest bridge-addable class containing this graph. 
Then $\hat {\mathcal{F}}=\cup_{n\geq 1}\hat{\mathcal{F}_n}$ is a tight  bridge-addable class. Nevertheless, the diameter of the random graph $\hat{F_n}$ is at least $n^{2/3}$ with probability $1$, while the diameter of the largest tree in a uniform random $n$-vertex forest is of order $\sqrt{n}$. Moreover, when properly renormalized by a scaling factor of $n^{-2/3}$, $\hat{F_n}$ converges for the GH topology to a real interval and not to the CRT.
However,  it may be true in general that \emph{typical} distances in tight bridge-addable classes are of order $\sqrt{n}$, and even that some convergence to the CRT holds if one allows that a (random) subset of vertices of size $o(n)$ is removed from the graph. We were not able to construct a counterexample, and thus leave these questions open.
\end{rem}

We conclude this list of results with a simpler statement that does not require the full strength of our main theorems (it is a relatively easy consequence of the results of~\cite{CP15}, and we will prove it in Section~\ref{sec:first}).
\begin{thm}\label{thm:components}
Let $\mathcal{G}$ be a tight bridge-addable class and $G_n$ a uniform random graph from $\mathcal{G}_n$. Then for any $k\geq 0$, we have 
$$
\Pr\left(G_n \mbox{ has $k+1$ connected components}\right)
\longrightarrow
e^{-1/2}\;\frac{2^{-k}}{k!}.
$$
In other words, the number of connected components of $G_n$ converges in distribution to $1+Poisson(1/2)$.
\end{thm}

\noindent{\bf Structure of the paper.} 
The proof of Theorem~\ref{thm:main} occupies most of the paper, and we now briefly present its structure. 
The proof roughly follows the one of Theorem~\ref{thm:conj}, which we proved in~\cite{CP15}: very loosely speaking we show that for a class to be $\zeta$-tight, some form of tightness has to occur in each intermediate inequality proven in~\cite{CP15}. As the length of the present paper shows, there is however quite an important amount of work to be done to achieve this goal. In particular, apart from some key definitions and notation, and two statements that we directly import from~\cite{CP15} (Theorem~\ref{thm:conj} and Lemma~\ref{lemma:CP15Prop4}), the two papers are disjoint.

We start in Section~\ref{sec:first} by proving elementary results about the number of components (including Theorem~\ref{thm:components}) and we introduce some notions that will play a crucial role in the rest of the proof. Importantly, in Section~\ref{subsec:partition}, we introduce the partitioning of the space that underlies our technique of local double-counting from~\cite{CP15}. In particular we define the notion of ``box'' that we use in order to partition each graph class according to the local structure of the graphs it contains.

Sections~\ref{sec:trees} and~\ref{sec:transfer} occupy the most important part of the paper.
 In Section~\ref{sec:trees}, we prove an analogue of Theorem~\ref{thm:main} under the assumption that all elements of $\cG$ are forests. This is done in several steps: 
In~\ref{subsec:GoodAndBadBoxes} we define the notion of ``good boxes'' and we prove that most of the mass in tight bridge-addable graph classes is localized inside good boxes. These good boxes have the property that they realize locally the extremal value of the optimization problem introduced in~\cite{CP15}. This optimization problem expresses some ratios inherited from a double-counting strategy in terms of parameters that record the local structure of the graphs. In~\ref{subsec:stability} we study the stability of this problem and deduce that for good boxes, all parameters have to be close to the unique extremum value (closely related to the quantities $a_\infty$ and $p_\infty$ appearing in Theorem~\ref{thm:main}). In~\ref{ssc:k1} we use these facts to prove a version of our main result when the graph $G_n$ has one or two components. Finally in~\ref{ssc:k} we use an induction on the number of components to conclude the proof, in the case of forests.

In Section~\ref{sec:transfer}, we address the case of general bridge-addable graph classes: 
In~\ref{subsec:removable} we prove that $\zeta$-tight bridge-addable classes tend to have many removable edges (edges that when deleted from a graph in the class, give rise to a graph in the class), and in~\ref{subsec:conclude} we use this property and the results of Section~\ref{sec:trees} to conclude the proof of Theorem~\ref{thm:main}.
Finally, in~\ref{subsec:BS}, we give a detailed proof of Corollary~\ref{cor:BS}.

Appendix~\ref{subapp:components} recalls classical results on enumeration of forests and on random forests, while Appendix~\ref{subapp:example} gives more details about the example of Remark~\ref{rem:example}.

For the reader's convenience we include here a table of contents of the paper:
\newpage
\renewcommand{\contentsname}{\vspace{-6mm} }
\tableofcontents

\section{First results and set-up for the proof}
\label{sec:first}

In this section, we obtain our first results and we introduce important notions and notation used in the whole paper.
In~\ref{subsec:components} we study the number of connected components and we prove Theorem~\ref{thm:components}. In~\ref{subsec:partition}, we define the partitioning of the space that underlies our technique of local double-counting. Finally in \ref{subsec:quantifiers}, we give a few precisions for the use of quantifiers in the rest of the paper.

\subsection{Number of components in bridge-addable graph classes}
\label{subsec:components}

Through the rest of the paper, for a bridge-addable class of graphs $\mathcal{G}$ and for $i\geq 1$, we denote by $\mathcal{G}^{(i)}_n$ the set of $n$-vertex graphs in $\mathcal{G}$ having $i$ connected components.
An elegant double-counting argument going back to~\cite{MCSW} asserts that for all $i\geq 1$, and  $n\geq 1$ we have:
\begin{eqnarray}\label{eq:easyBound}
i \cdot \left|\mathcal{G}_n^{(i+1)}\right|
\leq
\left|\mathcal{G}_n^{(i)}\right|.
\end{eqnarray}
This statement follows by double-counting the edges of an auxiliary bipartite graph on the vertex set $\mathcal{G}_n^{(i)}\uplus\mathcal{G}_n^{(i+1)}$, where two graphs $G,H$ are linked by an edge if and only if one can be obtained from the other by adding a bridge: on the one hand, an element of $\mathcal{G}_n^{(i+1)}$ has degree at least $i(n-i)$ in this auxiliary graph, since $\cG$ is bridge-addable; on the other hand, an element of $\mathcal{G}_n^{(i)}$ has degree at most $(n-i)$ (which is the maximum number of cut-edges in a graph with $i$ connected components and $n$ vertices). Thus \eqref{eq:easyBound} follows. 
The main achievement of the paper~\cite{CP15} was to improve this bound by roughly a factor $\frac{1}{2}$, asymptotically.
\begin{lemmaalpha}[{\cite[Proposition 5]{CP15}}]\label{lemma:CP15Prop4}
 For every $\eta$ and every $m$, if $\cG$ is a bridge-addable class and $n$ is large enough,  one has for every $i\leq m$,
\begin{align}\label{eq:comp1}
i |\cG^{(i+1)}_{n}| \leq \left(\frac{1}{2}+\eta \right) |\cG^{(i)}_{n}|\;.
\end{align}
\end{lemmaalpha}
The following lemma, which follows relatively easily from Lemma~\ref{lemma:CP15Prop4}, provides a converse inequality to~\eqref{eq:comp1} for $\zeta$-tight classes. Note that it implies Theorem~\ref{thm:components}.
\begin{lemma}\label{lem:tight}
For every $\eta$ and every $m$ there exists a $\zeta$ such that for every $\zeta$-tight bridge-addable class $\cG$ and provided $n$ is large enough, one has for every $i\leq m$,
$$
 \left(\frac{1}{2}-\eta \right) |\cG^{(i)}_n|\leq i|\cG^{(i+1)}_n| \leq   \left(\frac{1}{2}+\eta \right) |\cG^{(i)}_n|\;.
$$
\end{lemma}
\begin{proof}
The second inequality is precisely Lemma~\ref{lemma:CP15Prop4}.

The path to prove the first inequality is rather straightfoward, but we will proceed carefully since the order in which the different parameters are chosen requires some caution. We proceed by contradiction. 
Fix $\eta$ and $m$ and assume that for every $\zeta>0$ there exist a $\zeta$-tight bridge-addable class $\cG$, a large enough $n_*$ and an $i_*\leq m$ such that 
\begin{align}\label{eq:compAbsurd}
i_*|\cG^{(i_*+1)}_{n_*}| \leq   \left(\frac{1}{2}-\eta \right) |\cG^{(i_*)}_{n_*}|\;.
\end{align}
Let $i_0\geq m$ be an integer that we will choose later.
By Lemma~\ref{lemma:CP15Prop4}, if $n$ is large enough,~\eqref{eq:comp1} holds with $\eta=\zeta$ for any $i\leq i_0$.
Also, since $\cG$ is $\zeta$-tight, provided that $n_*$ is large enough, we have
\begin{align}\label{eq:compTight}
\frac{|\cG^{(1)}_{n_*}|}{|\cG_{n_*}|}\leq (1+\zeta)e^{-1/2}\;.
\end{align}
Noting $f_i(x):=\sum_{j>i} \frac{x^j}{j!}$, 
we can now bound the inverse of the probability that $G_{n_*}$ is connected as follows 
\begin{align*}
\frac{|\cG_{n_*}|}{|\cG_{n_*}^{(1)}|}& \leq \sum_{i=1}^{i^*-1}\frac{|\cG^{(i)}_{n_*}|}{|\cG_{n_*}^{(1)}|}+\sum_{i=i_*}^{i_0}\frac{|\cG^{(i)}_{n_*}|}{|\cG_{n_*}^{(1)}|}+\sum_{i\geq i_0+1}\frac{|\cG^{(i)}_{n_*}|}{|\cG_{n_*}^{(1)}|}\\
&\leq \sum_{i=1}^{i^*-1}\frac{1}{i!}\left(\frac{1}{2}+\zeta\right)^i +\sum_{i=i_*}^{i_0} \frac{1}{i!}  \left(\frac{1}{2}+\zeta \right)^i \frac{\frac{1}{2}-\eta}{\frac{1}{2}+\zeta}+ f_{i_0}(1)
\end{align*}
where for the last term we used the bound~\eqref{eq:easyBound}.
Thus:
\begin{align*}
\frac{|\cG_{n_*}|}{|\cG_{n_*}^{(1)}|}&\leq e^{\frac{1}{2}+\zeta}-f_{i_0}(1/2+\zeta)+ \left(\frac{\frac{1}{2}-\eta}{\frac{1}{2}+\zeta}-1\right)(f_{i_*-1}(1/2+\zeta)-f_{i_0}(1/2+\zeta)) + f_{i_0}(1)\\
&\leq e^{\frac{1}{2}+\zeta}- \frac{\eta+\zeta}{1/2+\zeta}\cdot f_{i_*-1}(1/2) + f_{i_0}(1)\\
&\leq e^{1/2}+(e^{\zeta}-1)e^{1/2}- \eta f_{m}(1/2) + f_{i_0}(1)\\
\end{align*}
We now choose $\zeta$ small enough with respect to $\eta$ and $m$ such that $\frac{\eta}{2} f_{m}(1/2)\geq (e^{\zeta}-1 + 2\zeta)e^{1/2}$, 
and we choose $i_0$ large enough with respect to $m$, in such a way that $\frac{\eta}{2} f_{m}(1/2)\geq f_{i_0}(1)$. These choices fix the value $n_*$ as above, and we finally get the bound:
\begin{align*}
\frac{|\cG^{(1)}_{n_*}|}{|\cG_{n_*}|}& \geq (1-2\zeta)^{-1} e^{-1/2} \geq (1+2\zeta) e^{-1/2}\;,
\end{align*}
However, since $n_*$ is arbitrarily large, we obtain a contradiction with~\eqref{eq:compTight}.
\end{proof}

\subsection{Partitioning the graph class into highly structured subclasses}
\label{subsec:partition}

We now introduce a partitioning of $\cG_n$ in terms of some local statistics, which requires the following set-up, that is modeled on~\cite[proof of Prop~3]{CP15}. Here $\epsilon$ and $k_*$ are two constants, whose value may vary along the course of the paper, that will \textit{in fine} be chosen very small and very large, respectively:
\begin{itemize}
\item[-]  $\cU_\epsilon$ is the set of unrooted trees of order at most $\lceil \epsilon^{-1}\rceil$:
$$
\cU_\epsilon:=\{U\in \mathcal{U},\, |U|\leq \lceil \epsilon^{-1}\rceil\}.
$$
\item[-]  $\cT_*$ is the set of rooted trees of order at most $k_*$:
$$
\cT_*:=\{T\in \mathcal{T},\, |T|\leq k_*\}.
$$
\end{itemize}
More generally, for any given $\ell\geq 1$, we will use the notation $\cT_{\leq \ell}$ (resp., $\cU_{\leq \ell}$) to denote the set of rooted (resp., unrooted) trees of order at most $\ell$. Then, $\cU_\epsilon=\cU_{\leq\lceil\epsilon^{-1}\rceil}$ and $\cT_*=\cT_{\leq k_*}$.

\smallskip
Roughly speaking, we will use elements of $\cU_{\epsilon}$ and $\cT_*$ as ``test graphs'' to measure the shape of small components of $G_n$ and the number of pending subtrees of $G_n$ of given shapes, respectively.
For $\ell\geq 1$, we write $\cE_\ell=\{0,\dots,n-1\}^{\cT_{\leq \ell}}$, and we will be particularly concerned with the set $\cE_{k_*}$, namely the set of integer vectors with one coordinate for each ``test tree'' in $\cT_*$.
For $\alpha\in \cE_{k_*}$ and $w\geq 1$ (width), we define the \emph{box} $[\alpha]^w\!\subset \cE_{k_*}$ and its \textit{$q$-neighborhood} $[\alpha]^w_q$
as the parallelepipeds:
\begin{align*}
[\alpha]^w&:=\{\alpha'\in \cE_{k_*}: \ \forall T \in \cT_*, \ \alpha(T) \leq \alpha'(T) < \alpha(T)+w\}\;,\\
[\alpha]^w_q&:=\{\alpha'\in \cE_{k_*}: \ \forall T \in \cT_*, \ \alpha(T)-q \leq \alpha'(T) < \alpha(T)+w+q\}\;.
\end{align*}
Note that here, and elsewhere in the paper, we slightly abuse notation by using both the letter $\alpha$ to denote an element of $\mathcal{E}_{\ell}$ and the notation $\alpha^G$ to denote the function $\alpha^G:\mathcal{T}\rightarrow \mathcal{E}_{\ell}$ that counts the number of pendant trees of a given shape in the graph $G$.

Finally, if $\mathcal{S}_n$ denotes a set of graphs (where the letter $\mathcal{S}$ could carry other decorations), we let $\mathcal{S}_{n,[\alpha]^w}$ be the set of graphs $G$ in $\mathcal{S}_n$ such that $(\alpha^G(T))_{T\in\cT_*} \in [\alpha]^w$, and we use the same notation with $[\alpha]^w_q$.

Also, for every forest $\{U_1,\dots,U_k\}$, we denote by ${\mathcal{S}_n}^{\{U_1,\dots,U_k\}}$ the set of graphs $G$ in $\mathcal{S}_n$ such that $\sma(G)$ is isomorphic to $\{U_1,\dots U_k\}$. While we denote a forest by $\{U_1,\dots, U_k\}$, one should understand it as an \emph{unordered multiset} of unrooted trees.
In the case of graphs with two connected components, we just use the notation $\mathcal{S}_n^U$ for $\mathcal{S}_n^{\{U\}}$, where $U\in \cU$.

\subsection{Notation and quantifiers in the proof}
\label{subsec:quantifiers}

The proof of Theorem~\ref{thm:main} involves many different quantifiers. 
In the statements and everywhere in the paper, all Greek letters, apart from $\alpha$ and $\beta$,
are implicitly assumed (in addition to other written quantifications or assumptions) to be positive constants that are smaller than $c$ where $c$ is a small, absolute, constant. The letters $\alpha$ and $\beta$ are used to refer to elements of the space $\cE_{k*}$ or $\cE_\ell$. We also use latin letters to denote integers that are greater than or equal to one.

Each statement in Sections~\ref{sec:trees} and~\ref{sec:transfer} involves several variables and the relative dependency between them plays a subtle role in the proof. We have carefully made all quantifiers explicit in all the statements.
However, the reader can use the following inequalities to clarify the hierarchy of (small) parameters used in Sections~\ref{sec:trees} and~\ref{sec:transfer}:
\begin{align}\label{eq:inequalities}
\frac{1}{n}\ll
\zeta\ll
\frac{1}{w}\ll
\frac{1}{k_*}\ll
\xi\ll
\epsilon=\frac{1}{q}\ll
\gamma \ll
\rho\ll
\nu \ll
\vartheta\ll
\eta\ll
\theta_1\ll
\dots\ll
\theta_k,\delta,1/\ell,1/k,1/u \leq 1,
\end{align}
where the notation $a\ll b\leq 1$ has to be read as: 
In each statement involving both variables $a$ and $b$, there exists a non-decreasing function $f:(0,1]\to (0,1]$ such that the statement holds for every $0<a\leq b\leq 1$ such that $a\leq f(b)$.
For example, the order in which the quantifiers appear in the statement of Lemma~\ref{lem:tight} above correspond to the notation:
$$
n^{-1}\ll\zeta\ll\eta, m^{-1}.
$$
Note that $1/n$ is the leftmost quantity appearing in~\eqref{eq:inequalities}: throughout the paper, $n$ will be taken arbitrarily large with respect to all the other constants. We will often write ``for $n$ large enough'' rather than ``there exists $n_0$ such that for $n\geq n_0$...'', and it will always be the case that the implicit value of the parameter $n_0$ is chosen after all other parameters, and may depend on all of them.

During the proof, we will use the notation $a=b\pm\mu$ to denote that $b-\mu\leq a\leq b+\mu$.

\subsection{Evaluation of generating functions of trees and forests}
\label{subsec:evalGF}

In this subsection we recall two classical evaluations of generating functions of trees and forests that we will use several times in our proofs. We let $T(z)=\sum_{n\geq 1} \frac{t_n}{n!}z^n$ be the exponential generating function of all rooted labelled trees, where the exponent of the variable $z$ records the number of vertices. Hence $t_n$ is the number of rooted trees on $[1..n]$, given by Cayley's formula: $t_n=n^{n-1}$. We also let $F(z)=\sum_{n\geq 0} \frac{f_n}{n!}z^n$ be the exponential generating function of (unrooted) labelled forests (here $f_n$ is the number of forests on the vertex set $[1..n]$, and by convention $f_0=1$). We have:
\begin{lemma}\label{lemma:evalGF}
Both $T(z)$ and $F(z)$ have radius of convergence $e^{-1}$, and both are finite at their main singularity $z=e^{-1}$, where we have $T(e^{-1})=1$ and $F(e^{-1})=e^{1/2}$. Moreover for $z$ in a slit neighbourhood of $e^{-1}$ we have
\begin{align}\label{eq:evalT}
T(z)  &= 1 + O(\sqrt{1-ze})\;.
\end{align}
\end{lemma}
The proof is a classical exercise in analytic combinatorics and is recalled in Appendix~\ref{app:gf}.

\section{Theorem~\ref{thm:main} for bridge-addable classes of forests}\label{sec:trees}

Balister, Bollob\'as and Gerke~{\cite[Lemma~2.1]{BBG}} proposed an elegant argument that reduces the statement of Theorem~\ref{thm:conj} to the case where all graphs in $\mathcal{G}$ are forests. 
As we will see in the next section, their idea can be adapted to the present context.
We will therefore start by proving Theorem~\ref{thm:main}  for classes $\cG$ composed by forests:
\begin{quote}
\it
Throughout the rest of Section~\ref{sec:trees}, we will assume that all graphs in $\cG$ are forests.
\end{quote}

\subsection{Good and bad boxes}
\label{subsec:GoodAndBadBoxes}
The main concern of the paper~\cite{CP15} was to obtain a version of the double-counting argument of Section~\ref{subsec:components} that is \emph{local} in the sense that it relates cardinalities of graphs corresponding to fixed boxes.

In order to select a collection of boxes, we will focus on the graphs in $\cG_n$ that have either one or two connected components, and, in view of this, we use the shorter notation $\mathcal{A}_n:=\cG^{(1)}_n$ and $\mathcal{B}_n:=\cG^{(2)}_n$.

Given $\epsilon$ (hence $\cU_\epsilon$) and $k_*$ (hence $\cT_*$), \cite[Lemma 17]{CP15} asserts that there exist integers $K$ and $w$ 
(independent of $\mathcal{G}$ and of $n$) and a set of $K$ disjoint boxes of width $w$ in~$\cE_{k_*}$, noted $\{[\beta_i]^w,\, 1\leq i \leq K\}$, such that if $q=q_\epsilon:=\lceil \epsilon^{-1}\rceil$ and if $n$  is large enough, 
then the $q$-neighbourhoods of boxes form a partition  of $\cE_{k_*}$,
\begin{align}\label{eq:Ppartition}
\biguplus_{i=1}^K  [\beta_i]^w_q = \cE_{k_*},
\end{align}
and moreover for each $U\in\cU_\epsilon$ we have:
\vspace{-2mm}
\begin{align}\label{eq:P1}
\sum_{i=1}^K  |\mathcal{B}^U_{n,[\beta_i]^w}| \geq  (1-\epsilon) |\mathcal{B}_n^{U}|. 
\end{align}
Note that from~\eqref{eq:Ppartition}, the boxes $[\beta_i]^w$ are $2q$-apart from each other, and yet~\eqref{eq:P1} ensures that they capture a proportion at least $(1-\epsilon)$ of the set $\mathcal{B}_n^U$ for each $U\in \cU_\epsilon$.
We now fix such a set of boxes, and we will keep referring to these  boxes $([\beta_i]^w)_{1\leq i \leq K}$ (or simply,  $([\beta_i])_{1\leq i \leq K}$) throughout Section~\ref{sec:trees}, keeping in mind that the number $K=K(\epsilon,k_*)$ of boxes, and their width $w=w(\epsilon, k_*)$, depend on $\epsilon$ and $k_*$ but neither on $\mathcal{G}$ nor on $n$.

In the present paper, one of the main tasks consists in showing that the global estimates obtained in~\cite{CP15}, such as Lemma~\ref{lemma:CP15Prop4}, can be ``lowered'' down to boxes for $\zeta$-tight classes. This is not true for every box in $\cE_{k_*}$, but it will be for certain boxes that contain most of the graphs in the class. For every $\gamma$ and every $\epsilon$, we say that a box $[\alpha]^w$ is \emph{$(\gamma,\epsilon)$-good} (or simply \emph{good}) if the two following conditions hold:
\begin{itemize}
\item[i)] $ |\cB_{n,[\alpha]^w}|\geq \left(\frac{1}{2}-\gamma\right)\cdot |\cA_{n,[\alpha]^w_{q}}| $, and
\item[ii)]
$\sum_{U\not\in\cU_\epsilon} |\cB^U_{n,[\alpha]^w}| < \gamma |\cB_{n,[\alpha]^w}|\;.$
\end{itemize}
Note that Property i) is a local version of the first inequality of Lemma~\ref{lem:tight} for $i=1$, while Property~ii) ensures that the number of graphs in sets that we do not control, is small.
 Hence good boxes are, in some sense, boxes that realize the tightness property locally. 

We will be interested in the boxes among the $[\beta_i]$ that are $(\gamma,\epsilon)$-good:
$$
\Good_{\gamma,\epsilon}:=\{i\in [1..K]:\, [\beta_i] \text{ is $(\gamma,\epsilon)$-good}\}\;.
\vspace{-1.5mm}
$$
An important step in the proof of Theorem~\ref{thm:main} is the following result:
\vspace{-1.5mm}
\begin{lemma}\label{lem:bad_cells}
For every $\gamma$ and every $\eta$, if $\epsilon<\epsilon_0(\gamma,\eta)$ and if $k_*\geq k_0(\epsilon)$, then there exists $\zeta$ such that for every $\zeta$-tight bridge-addable class $\cG$ and every large enough $n$, one has
$$
\frac{\sum_{i\notin \Good_{\gamma,\epsilon}} |\cA_{n,[\beta_i]^w_q}|}{|\cA_n|}<\eta\;,
$$
and
$$
\frac{\sum_{i\notin \Good_{\gamma,\epsilon}} |\cB_{n,[\beta_i]^w}|}{|\cB_n|}<\eta\;.
\vspace{-3mm}
$$
\end{lemma}
\begin{proof}
Let $\epsilon>0$ (to be fixed later). Up to setting $k_*$ and $n$ large enough, we can use Equation~(16) in~\cite{CP15} for each $1\leq i\leq K$,
\begin{align*}
\sum_{U\in\cU_\epsilon}|\cB^U_{n,[\beta_i]^w}|\leq
 \frac{1}{2}\cdot |\cA_{n,[\beta_i]^w_q}| (1+3\epsilon) \leq
 \left(\frac{1}{2}+2\epsilon\right)\cdot |\cA_{n,[\beta_i]^w_q}|\;.
\end{align*}
Moreover, provided that $n$ is large enough, we have (Equation~(17) in~\cite{CP15})
\begin{align}\label{eq:ub_largeU}
\sum_{U\not\in\cU_\epsilon}|\cB^U_{n}| \leq 2\epsilon \left|\mathcal{A}_n\right|
\end{align}
From the last two inequalities, we have
\begin{align}\label{eq:ub_local}
\sum_{i \in \Good_{\gamma,\epsilon}}|\cB_{n,[\beta_i]^w}|
\leq
2\epsilon\left|\mathcal{A}_n\right|+ 
\left(\frac{1}{2}+2\epsilon\right)\sum_{i \in \Good_{\gamma,\epsilon}}|\cA_{n,[\beta_i]^w_q}|\;.
\end{align}
Let $S$ and $T$ be the sets of indices $i\not\in\Good_{\gamma,\epsilon}$ such that $[\beta_i]^w$ violates i) and ii) respectively. Using~\eqref{eq:ub_largeU}, we have
\begin{align*}
 \sum_{i\in T}  |\cB_{n,[\beta_i]^w}|\leq  \sum_{i\in T}  \frac{1}{\gamma}\sum_{U\not\in\cU_\epsilon} |\cB^U_{n,[\beta_i]^w}| \leq\frac{1}{\gamma}\sum_{U\not\in\cU_\epsilon}|\cB^U_{n}|\leq \frac{2\epsilon}{\gamma} |\cA_n|\;.
\end{align*}
From the previous equation it follows that
\begin{align}\label{eq:ub_nongood}
\sum_{i \not\in \Good_{\gamma,\epsilon}}  |\cB_{n,[\beta_i]^w}|
&\leq
 \sum_{i\in S}  |\cB_{n,[\beta_i]^w}| +\sum_{i\in T} |\cB_{n,[\beta_i]^w_q}|\nonumber\\
&\leq
\left(\frac{1}{2}-\gamma\right)\sum_{i\notin \Good_{\gamma,\epsilon}} |\cA_{n,[\beta_i]^w_q}| + \frac{2\epsilon}{\gamma} \left|\mathcal{A}_n\right|\;.
\end{align}
Using~\eqref{eq:ub_local}  and~\eqref{eq:ub_nongood}, we get
\begin{align*}
(\gamma+2\epsilon)\sum_{i\notin \Good_{\gamma,\epsilon}} |\cA_{n,[\beta_i]^w_q}|
\leq& 
(\gamma+2\epsilon)\sum_{i\notin \Good_{\gamma,\epsilon}} |\cA_{n,[\beta_i]^w_q}|+
\sum_{i\notin \Good_{\gamma,\epsilon}} |\cB_{n,[\beta_i]^w}|\\
&+\sum_{i\in \Good_{\gamma,\epsilon}} |\cB_{n,[\beta_i]^w}|
-\sum_{i=1}^K |\cB_{n,[\beta_i]^w}|\\
& \leq  (\gamma+2\epsilon)\sum_{i\notin \Good_{\gamma,\epsilon}} |\cA_{n,[\beta_i]^w_q}|+
\left(\frac{1}{2}-\gamma\right)\sum_{i\notin \Good_{\gamma,\epsilon}} |\cA_{n,[\beta_i]^w_q}|\\
&+\left(\frac{1}{2}+2\epsilon\right)\sum_{i\in \Good_{\gamma,\epsilon}} |\cA_{n,[\beta_i]^w_q}|+\frac{4\epsilon}{\gamma}\left|\mathcal{A}_n\right|-\sum_{i=1}^K |\cB_{n,[\beta_i]^w}|
\end{align*}
The last inequality can be simplified as
\begin{align}
(\gamma+2\epsilon)\sum_{i\notin \Good_{\gamma,\epsilon}} |\cA_{n,[\beta_i]^w_q}|
\leq
& \left(\frac{1}{2}+2\epsilon\right) \sum_{i=1}^K |\cA_{n,[\beta_i]^w_q}| -\sum_{i=1}^K |\cB_{n,[\beta_i]^w}|+\frac{4\epsilon}{\gamma}\left|\mathcal{A}_n\right| \nonumber \\
&\leq \left(\frac{1}{2}+\frac{6\epsilon}{\gamma}\right) \left|\mathcal{A}_n\right| -\sum_{i=1}^K |\cB_{n,[\beta_i]^w}| \label{eq:ub_int}
\end{align}
where we used that the $[\beta_i]^w_q$ are disjoint. 
Using~\eqref{eq:ub_largeU} again and \eqref{eq:P1}, we have
\begin{align*}
\sum_{i=1}^K |\cB_{n,[\beta_i]^w}|
&\geq \sum_{i=1}^K \sum_{U\in \cU_\epsilon}  |\cB^U_{n,[\beta_i]^w}|\\
&\geq (1-\epsilon) (|\cB_n|- 2\epsilon |\cA_n|)\\
&\geq (1-\epsilon) |\cB_n|-  2\epsilon |\cA_n|\\
\end{align*}
Finally, Lemma~\ref{lem:tight} with $i=1$ and $\eta$ replaced by $\epsilon$, implies that if $\zeta$ is small enough,  $\cG$ is $\zeta$-tight and $n$ is large enough, then the last quantity is larger than
$
(1/2-4\epsilon)\left|\mathcal{A}_n\right|.
$

We now choose $\epsilon_0= \frac{\eta \gamma}{20}$. Going back to \eqref{eq:ub_int}, if $\epsilon<\epsilon_0$, we get
\begin{align}\label{eq:first_part}
\sum_{i\notin \Good_{\gamma,\epsilon}} |\cA_{n,[\beta_i]^w_q}|
\leq \frac{10\epsilon}{\gamma(\gamma+2\epsilon)} \left|\mathcal{A}_n\right|\leq \frac{\eta}{2}|\cA_n|\;,
\end{align}
which proves the first part of the lemma.

For the second part of the lemma, we use~\eqref{eq:ub_nongood} and Lemma~\ref{lem:tight} with $\eta$ replaced by $\epsilon$, to get
\begin{align*}
\frac{\sum_{i\notin \Good_{\gamma,\epsilon}} |\cB_{n,[\beta_i]^w}|}{|\cB_n|}&\leq \frac{\left(\frac{1}{2}-\gamma\right)\sum_{i\notin \Good_{\gamma,\epsilon}} |\cA_{n,[\beta_i]^w_q}|+\frac{2\epsilon}{\gamma}|A_n|}{\left(\frac{1}{2}-\epsilon\right)|\cA_n|}\\
\end{align*}
By~\eqref{eq:first_part}, we conclude
\begin{align*}
\frac{\sum_{i\notin \Good_{\gamma,\epsilon}} |\cB_{n,[\beta_i]^w}|}{|\cB_n|}
&\leq \frac{\left(\frac{1}{2}-\gamma\right)\frac{\eta}{2}|\cA_n|+\frac{2\epsilon}{\gamma}|A_n|}{\left(\frac{1}{2}-\epsilon\right)|\cA_n|}
\leq \eta\;.\qedhere
\end{align*}
\end{proof}

\subsection{Stability of the extremum for the optimization problem}\label{ssc:stability}
\label{subsec:stability}
The goal of this subsection is to estimate the ratio between  $|\cB^U_{n,[\alpha]^w}|$ and $|\cA_{n,[\alpha]^w_q}|$, when $[\alpha]^w$ is a good box and $U\in \cU_\epsilon$. Recall that good boxes are the ones that locally inherit the global property of being tight; for instance, in good boxes we have a precise estimation of the ratio between $ |\cB_{n,[\alpha]^w}|$ and $|\cA_{n,[\alpha]^w_q}|$.

In order to do that, we will need to return to the original ``optimization problem'' introduced in~\cite{CP15}. Namely, we will study certain functionals of the ratios $|\cB^U_{n,[\alpha]^w}|/|\cA_{n,[\alpha]^w_q}|$, or more precisely of the variables $(z_{n,\alpha}^U)_{U\in \cU_\epsilon}$, defined by~\eqref{eq:defzn} below. 
We will proceed as follows: Lemma~\ref{lem:z_in_D} gives the ``constraints'' of the optimization problem, by showing that the variables $z_{n,\alpha}^U$ have to be close to a certain domain $D$; Lemma~\ref{lem:good_implies_Y_large} shows that if $[\alpha]^w$ is good, then the ``objective function'' of the optimization problem has to be close to its optimal value given these constraints (which was proved to be $\frac{1}{2}$ in~\cite{CP15}).
Then Lemma~\ref{lem:omega_controled} proves a form of unicity of the extremum. From these three lemmas we deduce the main results of this subsection: if $[\alpha]^w$ is good, then $(z_{n,\alpha}^U)_{U\in \cU_\epsilon}$ is close to $p_\infty(U)$ for each unrooted tree $U$ of bounded size~(Proposition~\ref{lem:1}) and if $[\alpha]^w$ is good, then $\alpha(T)/n$ is close to $a_\infty(T)$ for each rooted tree $T$ of bounded size~(Proposition~\ref{lem:2})

Apart from the proof of Lemma~\ref{lem:bad_cells} already given, the proofs of Lemmas~\ref{lem:z_in_D}--\ref{lem:good_implies_Y_large}--\ref{lem:omega_controled} are the part of the present paper that rely the most on \cite{CP15}. Indeed, we will refer to several technical statements therein in our proofs. This will no longer be the case in the next sections.

Following~\cite{CP15}, given $\epsilon$ (hence $\cU_\epsilon$) we define a \emph{$\cU_\epsilon$-admissible decomposition  of $T$} as an increasing sequence $\mathbf{T}=(T_i)_{i\leq \ell}$ of labeled trees $$T_1\subset \dots \subset T_\ell =T$$ for some $\ell \geq 1$ called the \emph{length}, such that $T_1\in \mathcal{U}_\epsilon$ and, for each $2\leq i \leq\ell$, $T_{i}$ is obtained by joining $T_{i-1}$ by an edge $e_i$ to some tree $U_i\in \cU_\epsilon$.
 The \emph{weight} of $\mathbf{T}$ with respect to $\mathbf{z}=(z^U)_{U\in\cU_\epsilon}\in (\mathbb{R}_+)^{\cU_\epsilon}$ is defined as $\omega (\mathbf{T}, \mathbf{z}) =\prod_{i=1}^{\ell} z^{U_i}$, where $U_i=T_i\setminus T_{i-1}$ as an unrooted tree~(here we use the convention $T_0=\emptyset$).  The \emph{maximum weight} of $T$ with respect to $\mathbf{z}$, denoted by $\omega(T,\mathbf{z})$, is defined as the maximum of $\omega (\mathbf{T}, \mathbf{z})$ over all the $\cU_\epsilon$-admissible decompositions $\mathbf{T}$ of~$T$.

We now use $\omega(T,\mathbf{z})$ to define the following partition functions,
\begin{align*}
Y(\mathbf{z}) := \sum_{T\in \cT} \frac{\omega(T,\mathbf{z})}{\mathrm{Aut}_r(T)}, & \qquad
Y^u(\mathbf{z}) := \sum_{U\in \cU} \frac{\omega(U,\mathbf{z})}{\mathrm{Aut}_u(U)}\;,
\\
Y_{\cT_*}(\mathbf{z}) := \sum_{T\in \cT_*} \frac{\omega(T,\mathbf{z})}{\mathrm{Aut}_r(T)}, & \qquad
Y^u_{\cU_\epsilon}(\mathbf{z}) := \sum_{U\in \cU_\epsilon} \frac{\omega(U,\mathbf{z})}{\mathrm{Aut}_u(U)}\\
Y_{\leq k}(\mathbf{z}):= \sum_{T\in \cT_{\leq k}} \frac{\omega(T,\mathbf{z})}{\mathrm{Aut}_r(T)}& \qquad
\tilde{Y}^u_{\cU_\epsilon}(\mathbf{z}) := \sum_{U\in \cU_\epsilon} \frac{z^U}{\mathrm{Aut}_u(U)}\;.
\end{align*}
Furthermore, we define the domain of convergence of $Y(\mathbf{z})$ as follows,
$$
D:=\{\mathbf{z} \in (\mathbb{R}_+)^{\cU_\epsilon}, Y(\mathbf{z})<\infty\}\;.
$$
It is important to note that there is an implicit dependence of $\omega(T,\mathbf{z})$ on $\epsilon$ (via $\cU_\epsilon$-admissible decompositions). Hence, all the partition functions defined above (and their respective domains) also depend on $\epsilon$. In order to keep the notation light we do not make this dependence explicit.

Let $\mathbf{j}:=(1)_{U\in \cU_\epsilon}$ be the all-one vector of length $|\cU_\epsilon|$. Given a choice of $n$, to each $\alpha\in \cE_{k_*}$ we assign a vector  $\mathbf{z}_{n,\alpha}= (z_{n,\alpha}^U)_{U \in \cU_\epsilon} \in (\mathbb{R}_+)^{\cU_\epsilon}$, where
\begin{align}\label{eq:defzn}
z_{n,\alpha}^U:=\Aut_u(U) \frac{|\mathcal{B}^U_{n,[\alpha]^w}|}{|\mathcal{A}_{n,[\alpha]^w_{q}}|} \left(1-\frac{|U|}{n}\right)\;,
\end{align}
where $q = \lceil \epsilon^{-1}\rceil$ as before and $w=w(\epsilon,k_*)$ is chosen as in Section~\ref{subsec:GoodAndBadBoxes}.

\begin{lemma}\label{lem:z_in_D}
For every $\xi$ and every $\epsilon$,  if $k_* \geq  k_0(\epsilon,\xi)$ and $n$ is large enough, then for every $\alpha\in \cE_{k_*}$ we have that $\mathbf{z}_{n,\alpha}-\xi \mathbf{j}\in D$.
\end{lemma}
\begin{proof}
For the sake of contradiction, assume that there exist $\xi$ and $\epsilon$ such that for every $k_0$ there exists $k\geq k_0$ such that for every large enough $n$ there exists $\alpha_{n,k}\in \cE_k$ with
$$
\mathbf{z}_{n,\alpha_{n,k}} -\xi \mathbf{j}\notin D\;.
$$

For a given $k\geq k_0$, let $\mathbf{z}_k$ be a limit point of the sequence $(\mathbf{z}_{n,\alpha_{n,k}})_{n\geq 1}$.  Since $D$ is closed downwards (Lemma~13 in~\cite{CP15}), then $\mathbf{z}_{k} -\frac{\xi}{2} \mathbf{j}\notin D$.

Moreover, by Corollary~12 in~\cite{CP15}, we have $Y_{\leq k}(\mathbf{z}_k)\leq 1$. Similarly as in \cite[Lemma~16]{CP15}, this implies that any limit point $\mathbf{z}_\infty$ of $(\mathbf{z}_k)_{k\geq k_0}$ satisfies $\mathbf{z}_\infty\in\overline{D}$. This is a contradiction with the fact that $\mathbf{z}_{k} -\frac{\xi}{2} \mathbf{j}\notin D$ for every $k\geq k_0$. 
\end{proof}

The following lemma shows that if $[\alpha]^w$ is $(\gamma,\epsilon)$-good, then the evaluation of $Y^u$ in a point close to $\mathbf{z}_{n,\alpha}$ is close to $\tfrac{1}{2}$ (which was shown in~\cite{CP15} to be the maximum of $Y^u$ on $D$).
\begin{lemma}\label{lem:good_implies_Y_large}
For every $\rho$, every $\epsilon$ and every $\ell$ such that $\ell<1/\epsilon$,
if $\gamma\leq \gamma_0(\rho,\ell)$, $\xi\leq \xi_0(\rho,\epsilon,\ell)$, $k_*\geq k_0(\epsilon,\xi)$ and $n$ is large enough, then
for every box $[\alpha]^w$ which is $(\gamma,\epsilon)$-good the following holds for $\hat{\mathbf{z}}:=\mathbf{z}_{n,\alpha}-\xi \mathbf{j}$: we have $\hat{\mathbf{z}}\in D$,
 $$
 Y^u(\hat{\mathbf{z}})>\frac{1}{2}-\rho\;,
 $$
and for every $U\in \cU_{\leq \ell}$, we have
$$
|\omega(U,\hat{\mathbf{z}})-\hat{z}^{U}|\leq \rho\;.
$$ 
\end{lemma}
\begin{proof} Let $\gamma_0:=\frac{\rho}{4\ell!}$ and $\xi_0:=\frac{\rho}{2|\cU_\epsilon|\ell!}$.
Consider $\alpha \in \cE_{k_*}$ such that the box $[\alpha]^w$ is $(\gamma,\epsilon)$-good. Using the properties~i) and ~ii) of good boxes, and~\eqref{eq:defzn}, we have
\begin{align*}
\tilde{Y}^u_{\cU_\epsilon}(\mathbf{z}_{n,\alpha}) 
=\sum_{U\in\cU_\epsilon} \frac{z^U_{n,\alpha}}{\Aut_u(U)} 
&= \frac{1}{|\cA_{n,[\alpha]^w_q}|}\sum_{U\in\cU_\epsilon}|\cB^U_{n,[\alpha]^w}|\left(1-\frac{|U|}{n}\right)\\ 
&\geq \frac{1}{|\cA_{n,[\alpha]^w_q}|}|\cB_{n,[\alpha]^w}|(1-\gamma)\left(1-\frac{|U|}{n}\right)\\ 
&\geq \left( \frac{1}{2}-\gamma\right)\left(1-\gamma\right)\left(1-\frac{1}{\epsilon n}\right)\\ 
&\geq \frac{1}{2}-2\gamma\;,
\end{align*}
provided that $n$ is large enough.
Now, since $\tilde Y^u_{\cU_\epsilon}(\mathbf{\hat z})$ is a finite sum, we have
 \begin{align*}
 \tilde Y^u_{\cU_\epsilon}(\hat{\mathbf{z}})& \geq \tilde Y^u_{\cU_\epsilon}(\mathbf{z}_{n,\alpha})- \xi |\cU_\epsilon|\;.
\end{align*}
Together with the previous inequality and the choice of $\gamma_0$ and $\xi_0$, this implies
\begin{align}\label{eq:tilde}
 \tilde Y^u_{\cU_\epsilon}(\hat{\mathbf{z}})\geq \frac{1}{2}- \left(\xi |\cU_\epsilon|+2\gamma\right)\geq \frac{1}{2} - \frac{\rho}{\ell!}\;.
 \end{align}
By definition of maximum weight, for every $U\in \cU_\epsilon$ we have $\omega(U,\mathbf{z})\geq z^U$, which directly implies $ Y^u_{\cU_\epsilon}(\mathbf{z})\geq  \tilde Y^u_{\cU_\epsilon}(\mathbf{z})$. We thus conclude the first part of the lemma,
$$
 Y^u(\hat{\mathbf{z}})\geq  Y^u_{\cU_\epsilon}(\hat{\mathbf{z}})\geq  \tilde Y^u_{\cU_\epsilon}(\hat{\mathbf{z}})\geq \frac{1}{2}- \frac{\rho}{\ell!}>\frac{1}{2}-\rho\;
$$
Observe that this is true even if $\hat{\mathbf{z}}\not\in D$, since then the LHS is infinite.
\smallskip

By Lemma~\ref{lem:z_in_D}, we can choose $k_0=k_0(\epsilon,\xi)$ such that if $k_*\geq k_0$ and $n$ is large enough, we have $\hat{\mathbf{z}}\in D$. The choice of $k_*$ and $n$ is suitable \emph{for all} vectors in $\cE_{k_*}$. Then, Lemma~14 in~\cite{CP15} implies that $Y^u_{\cU_\epsilon}(\hat{\mathbf{z}})\leq Y^u(\hat{\mathbf{z}})\leq \frac{1}{2}$. Together with~\eqref{eq:tilde}, for every $U\in\cU_\epsilon$ we have
$$
\frac{\rho}{\ell!} \geq |Y^u_{\cU_\epsilon}(\hat{\mathbf{z}})- \tilde{Y}^u_{\cU_\epsilon}(\hat{\mathbf{z}}) |= \left|\sum_{U'\in\cU_\epsilon}  \frac{\omega(U',\hat{\mathbf{z}})-\hat{z}^{U'}}{\Aut_u(U')}\right|\geq \frac{|\omega(U,\hat{\mathbf{z}})-\hat{z}^{U}|}{\Aut_u(U)}\;,
$$
where the last inequality follows since $\omega(U',\hat{\mathbf{z}})\geq \hat{z}^{U'}$ for each tree $U'\in \cU_\epsilon$.
Since $\Aut_u(U)\leq \ell!$, it follows that
\begin{align*}
&|\omega(U,\hat{\mathbf{z}})-\hat{z}^{U}|\leq \rho\;.\qedhere
\end{align*}
\end{proof}

The next lemma states that if $\mathbf{z}$ belongs to $D$ and  $Y^u(\mathbf{z})$ is close to $\tfrac{1}{2}$, then $\omega(T,\mathbf{z})$ is close to $e^{-|T|}$ for every $T$ with bounded size.
\begin{lemma}\label{lem:omega_controled}
 For every $\nu$ and every $\ell$, if $\rho\leq \rho_0(\nu,\ell)$, then for every $\epsilon$, every $\mathbf{z}\in D$ that satisfies $Y^u(\mathbf{z})> \frac{1}{2}-\rho$, and every $T\in \cT_{\leq \ell}$, we have
\begin{align}\label{eq:omega_controled}
 |\omega(T,\mathbf{z})-e^{-|T|}|<\nu\;.
\end{align}
\end{lemma}
\begin{proof}
Let $Y^e(\mathbf{z})$ be the partition function of trees rooted at an edge, where each tree is weighted by its maximal weight. As noted in~\cite{CP15}, a classical trick known as the dissymmetry theorem~\cite{quebecois} implies that \begin{align*}
Y^e(\mathbf{z})=Y(\mathbf{z})-Y^u(\mathbf{z}) \;.
\end{align*}
\noindent Together with the hypothesis of the lemma and the fact that $y-1/2\leq y^2/2$ for all $y\in \mathbb{R}$, this implies:
$$
Y^e(\mathbf{z})=Y(\mathbf{z})-Y^u(\mathbf{z})\leq Y(\mathbf{z})-1/2+\rho\leq \frac{1}{2}(Y(\mathbf{z}))^2+\rho\;,
$$
For every pair of vertex rooted trees $T_1,T_2\ \in \cT$, let $f(T_1,T_2)$ be the edge-rooted tree obtained by adding an edge (the root) connecting the roots of $T_1$ and $T_2$.
We have the following supermultiplicativity property:
$$
\omega(f(T_1,T_2),\mathbf{z}) - \omega(T_1,\mathbf{z}) \omega(T_2,\mathbf{z})\geq 0\;.
$$
Also observe that the number of automorphisms of $f(T_1,T_2)$ that fix the rooted edge (as an ordered edge!), is precisely $\Aut_r(T_1)\Aut_r(T_2)$. Thus, for any pair $R_1,R_2\ \in \cT$, we have
\begin{align}\label{eq:decomp_not_bad}
\rho\geq Y^e(\mathbf{z})-\frac{1}{2}(Y(\mathbf{z}))^2 &=\sum_{T_1,T_2\in \cT} \frac{
\omega(f(T_1,T_2),\mathbf{z}) - \omega(T_1,\mathbf{z})\omega(T_2,\mathbf{z})}{\Aut_r(T_1)\Aut_r(T_2)}\nonumber 
\\
&\geq \frac{ \omega(f(R_1,R_2),\mathbf{z}) - \omega(R_1,\mathbf{z})\omega(R_2,\mathbf{z})}{|R_1|! \; |R_2|!}\;.
\end{align}
Let $\bullet$ be the tree composed by a single vertex and define $x=x(\mathbf{z}):=\omega(\bullet,\mathbf{z})=z^\bullet \in\mathbb{R}_+$. Observe that since $\mathbf{z}\in D$, we have $x\leq 1$ (otherwise $Y(\mathbf{z})=\infty$ since $\omega(T,\mathbf{z})\geq x^{|T|}$). Using~\eqref{eq:decomp_not_bad} with $R_2=\bullet$, for every $T\in \cT$:
$$
\omega(f(T,\bullet),\mathbf{z}) \leq x \cdot \omega(T, \mathbf{z}) +\rho \cdot |T|!, 
$$
and induction on $|T|$ implies that for every $T\in\cT$ we have
$$
x^{|T|}\leq \omega(T,\mathbf{z})
\leq  x^{|T|}+ |T|!\rho
\leq \left(x+ (\rho  |T|!)^{\frac{1}{|T|}}\right)^{|T|}\; .
$$
Note that if $|T|\leq \ell$, then  $(\rho|T|!)^{\frac{1}{|T|}}\leq c(\ell) \rho^{\frac{1}{\ell}}$, for some $c(\ell)>0$.
Consider $\mathbf{x}=(x^{|U|})_{U\in \cU_\epsilon}$ and $\mathbf{x}_\rho=((x+c(\ell)\rho^{\frac{1}{\ell}} )^{|U|})_{U\in \cU_\epsilon}$. By the definition of $\mathbf{x}$, note that $\omega(T,\mathbf{x})= x^{|T|}$, therefore $\omega(T,\mathbf{x})\leq \omega(T,\mathbf{z})$ and since $\mathbf{z}\in D$, by Lemma 14 in~\cite{CP15} we have
$$
Y^u(\mathbf{x})\leq Y^u(\mathbf{z})\leq \frac{1}{2}\;.
$$
This implies $x\leq e^{-1}$ (otherwise $Y^u(\mathbf{x})$ would not converge).
Similarly $\omega(T,\mathbf{x_\rho})= (x+c(\ell)\rho^{\frac{1}{\ell}} )^{|T|}$, and using the hypothesis of the lemma we have
$$
\frac{1}{2}-\rho \leq Y^u(\mathbf{z})\leq Y^u(\mathbf{x}_\rho)\;.
$$
By Equation~\eqref{eq:evalT} in Lemma~\ref{lemma:evalGF},
 this implies that $x+c(\ell)\rho^{\frac{1}{\ell}} \geq e^{-1}-O(\sqrt{c(\ell)\rho^{1/\ell}})$. Given $\nu$ and $\ell$, we can now set $\rho_0(\nu,\ell)$ small enough such that for $\rho\leq\rho_0(\nu,\ell)$ one has $x> e^{-1}\left(1-y\right)$, with $y=\min\{\frac{\nu e^\ell}{\ell},1\}$, and $\rho\leq \frac{\nu}{\ell!}$. We then have, for every $T\in \cT_{\leq \ell}$,
$$
e^{-|T|}-\nu\leq  e^{-|T|}(1-y|T|) \leq e^{-|T|}\left(1-y\right)^{|T|}<  x^{|T|} \leq \omega(T,\mathbf{z}) \leq x^{|T|}+ \rho \cdot |T|! \leq e^{-|T|}+\nu\;.
$$
where we used that $(1-y)^{\ell}$ is convex for $y\in [0,1]$.
\end{proof}

Finally, we can prove estimates for the ratios between $|\cB^U_{n,[\alpha]^w}|$ and $|\cA_{n,[\alpha]^w_q}|$ for good boxes $[\alpha]^w$ and unrooted trees $U$ with bounded size.
\begin{prop}\label{lem:1}
For every $\vartheta$, every $\epsilon$ and every $\ell$ such that $\ell<1/\epsilon$, if $\gamma\leq \gamma_0(\vartheta,\ell)$, $k_*\geq k_0(\vartheta,\epsilon,\ell)$ and $n$ is large enough, then for every box $[\alpha]^w$ which is $(\gamma,\epsilon)$-good and every $U\in \cU_{\leq \ell}$
$$
\left|\frac{|\cB^U_{n,[\alpha]^w}|}{|\cA_{n,[\alpha]^w_q}|} -\frac{e^{-|U|}}{\Aut_u(U)}\right|<\vartheta\;.
$$
\end{prop}
\begin{proof}
Let us first fix the constants that we will need in the proof.
For $\nu:=\vartheta/4$, we let $\rho_0=\rho_0(\nu,\ell)$ be the value obtained from Lemma~\ref{lem:omega_controled}.
For $\rho:=\min\{\rho_0, \nu\}$, we let $\gamma_0=\gamma_0(\rho,\ell)$, $\xi_0=\xi_0(\rho,\epsilon,\ell)$ be the values obtained from Lemma~\ref{lem:good_implies_Y_large}.
For $\xi:=\min\left\{\xi_0,\nu\right\}$, we let $k_0=k_0(\epsilon,\xi)(=k_0(\vartheta,\epsilon,\ell))$ be the value obtained from Lemma~\ref{lem:good_implies_Y_large}.
Now fix $k_*\geq k_0$ and consider $n$ large enough. Note that once $k_*$ and $n$ are chosen, the space $\cE_{k_*}$ is well-determined.

Let $\hat{\mathbf{z}}=\mathbf{z}_{n,\alpha}-\xi \mathbf{j}$
as before.
For a given $U\in \cU_{\leq \ell}$, we observe
$$
|z^U_{n,\alpha}- \hat{z}^U|\leq \xi \leq \vartheta/4\;.
$$
By Lemma~\ref{lem:good_implies_Y_large}, if $[\alpha]^w$ is $(\gamma,\epsilon)$-good, we have
$$
|\hat{z}^{U} - \omega(U,\hat{\mathbf{z}})|\leq \rho \leq \vartheta/4\;.
$$   
The same lemma also implies that $\hat{\mathbf{z}}\in D$ and that $Y^u(\hat{\mathbf{z}})>\frac{1}{2}-\rho$.
Thus, $\hat{\mathbf{z}}$ satisfies the hypothesis of Lemma~\ref{lem:omega_controled}, which implies
$$
|\omega(U,\hat{\mathbf{z}})-e^{-|U|}|<\nu= \vartheta/4\;.
$$
Using the previous three inequalities and~\eqref{eq:defzn}, we conclude
\begin{align*}
\left|\frac{|\cB^U_{n,[\alpha]^w}|}{|\cA_{n,[\alpha]^w_q}|} -\frac{e^{-|U|}}{\Aut_u(U)} \right| 
&=     \frac{|z^U_{n,\alpha}\left(1-\frac{|U|}{n}\right)^{-1}-e^{-|U|}| }{\Aut_u(U)}<\vartheta\;,
\end{align*}
provided that $n$ is large enough. In the last inequality we used that $z_{n,\alpha}^U\leq 1$ (this can be obtained using a similar argument as the one used to obtain~\eqref{eq:easyBound}).
\end{proof}

\begin{prop}\label{lem:2}
For every $\vartheta$, every $\epsilon$ and every $\ell$ such $\ell<1/\epsilon$, if $\gamma\leq \gamma_0(\vartheta,\ell)$, $k_*\geq k_0(\vartheta,\epsilon,\ell)$ and $n$ is large enough, then for every box $[\alpha]^w$ which is $(\gamma,\epsilon)$-good and every $T\in \cT_{\leq\ell}$
$$
\left|\frac{\alpha(T)}{n}-\frac{e^{-|T|}}{\Aut_r(T)}\right|<\vartheta\;.
$$
\end{prop}
\begin{proof}

Again, let us start by fixing the constants that we will need in the proof. For $\nu:=\frac{\vartheta}{4|\cT_{\leq \ell}|}$, we let $\rho_0=\rho_0(\nu,\ell)$ be the value obtained from Lemma~\ref{lem:omega_controled}. For $\rho\leq \rho_0$, we let $\gamma_0=\gamma_0(\rho,\ell)(=\gamma_0(\vartheta,\ell))$, $\xi_0=\xi_0(\rho,\epsilon,\ell)$ be the values obtained from Lemma~\ref{lem:good_implies_Y_large}.

Observe that, if we fix $T\in \cT_{\leq \ell}$, the function $\omega(T,\mathbf{z})$ is a \emph{piecewise} polynomial in the set of variables $\{z^U:\,U\in \cU_\epsilon\}$ that it is continuous at every point of $(\mathbb{R}_+)^{\cU_\epsilon}$. Since $D$ is bounded, there exists a $\xi_1$ such that for every $\xi\leq \xi_1$ and every $\mathbf{z}$ at distance at most $1$ from $D$ (in the $\ell_\infty$ norm), we have
\begin{align*}
|\omega(T,\mathbf{z})-\omega(T,\mathbf{z}-\xi \mathbf{j})|<\frac{\vartheta}{4|\cT_{\leq \ell}|}\;.
\end{align*}
 
For $\xi:=\min\{\xi_0,\xi_1\}$, we let $k_0=k_0(\epsilon,\xi)(=k_0(\vartheta,\epsilon,\ell))$ be the value obtained from Lemma~\ref{lem:good_implies_Y_large}.
Fix $k_* \geq k_0$ and consider $n$ large enough.

By Lemma~\ref{lem:good_implies_Y_large}, 
if $[\alpha]^w$ is $(\gamma,\epsilon)$-good and we write  $\hat{\mathbf{z}}:=\mathbf{z}_{n,\alpha}-\xi \mathbf{j}$, we have $\hat{\mathbf{z}}\in D$ and
$Y^u(\hat{\mathbf{z}})>\frac{1}{2}-\rho$. Thus, $\hat{\mathbf{z}}$ satisfies the hypothesis of Lemma~\ref{lem:omega_controled} and we have
\begin{align*}
|\omega(T,\hat{\mathbf{z}})-e^{-|T|}|<\nu=\frac{\vartheta}{4|\cT_{\leq\ell}|}\;.
\end{align*}
Using the previous inequalities, we obtain
$$
|\omega(T,\mathbf{z}_{n,\alpha})-e^{-|T|}|\leq |\omega(T,\mathbf{z}_{n,\alpha})-\omega(T,\hat{\mathbf{z}})|+|\omega(T,\hat{\mathbf{z}})-e^{-|T|}|<\frac{\vartheta}{2|\cT_{\leq \ell}|} \;.
$$
By Lemma 11 in~\cite{CP15}, there exists a constant $C$ that does not depend on $n$ such that
\begin{align}\label{eq:LB}
\frac{\alpha(T)}{n} \geq \frac{\omega(T,\mathbf{z}_{n,\alpha})}{\Aut_r(T)} -\frac{C}{n}\geq  \frac{e^{-|T|}}{\Aut_r(T)}-\frac{2\vartheta}{3|\cT_{\leq\ell}|} \;,
\end{align}
where the last inequality holds provided $n$ is large enough. This proves one side of the inequality in the statement.

By Lemma~\ref{lemma:evalGF},
 if we let $t$ be large enough with respect to $\vartheta$, we have that
\begin{align}\label{eq:almost one}
\sum_{T\in\cT_{\leq t}}\frac{e^{-|T|}}{\Aut_r(T)} > 1-\frac{\vartheta}{3} .
\end{align}
We can assume that $\ell\geq t$, up to increasing the value of $k_*$ and $n$.

For the sake of contradiction, suppose that there exists $T_0\in \cT_{\le \ell}$ such that $\frac{\alpha(T_0)}{n} > \frac{e^{-|T_0|}}{\Aut_r(T_0)}+\vartheta$. Then, using~\eqref{eq:LB},~\eqref{eq:almost one} and the properties of $T_0$, we get
$$
1\geq \sum_{T\in \cT_{\leq\ell}}\frac{\alpha(T)}{n}\geq \sum_{T\in\cT_{\leq\ell}}\frac{e^{-|T|}}{\Aut_r(T)} -\frac{2\vartheta}{3}+\vartheta > 1\;,
$$
thus obtaining a contradiction and concluding the proof of the lemma.
\end{proof}

\subsection{Proof of Theorem~\ref{thm:main} for classes of forests: the case of $1$ or $2$ connected components}\label{ssc:k1}

For every $\delta$ and every $\ell$, consider the set of vectors in $\cE_\ell$ that are $\delta$-close from the distribution $a_\infty$ (recall that for $T\in \cT$, $a_\infty(T)=\frac{e^{-|T|}}{\Aut_r(T)}$); that is,
\begin{align}\label{def:Xi}
\Xi(\delta,\ell)&=\left\{\beta\in \cE_\ell:\, \left|\frac{\beta(T)}{n}-a_\infty(T)\right|<\delta,\text{ for every }T\in \cT_{\leq \ell}\right\}\;.
\end{align}

\noindent In what follows, for every set of graphs $\mathcal{S}_n$, every $\ell\geq 1$ and every $\beta\in \cE_\ell$, we let $\mathcal{S}_{n,\beta}$ be the set of graphs $G$ in $\mathcal{S}_n$ such that $\alpha^G(T) = \beta(T)$ for all $T\in\cT_{\leq \ell}$.
\begin{prop}\label{prop:k1}
For every $\theta_1$ and every $U \in \mathcal{U}$, there exists a $\zeta$ such that for every $\zeta$-tight class $\cG$ of forests and every large enough $n$, one has
\begin{align*}
 \left| \frac{ \left| \cB_n^{U} \right|}{|\cG_n|}- e^{-1/2}\frac{e^{-|U|}}{\Aut_u(U)}\right|< \theta_1\;.
\end{align*}
Moreover, for every $\theta_1$, every $\delta$, every $\ell$ and every $U \in \mathcal{U}$, there exists a $\zeta$ such that for every $\zeta$-tight class $\cG$ of forests and every large enough $n$, one has
\begin{align*}
\left| \frac{\sum_{\beta\in \Xi(\delta, \ell)}\left| \cB_{n,\beta}^{U} \right|}{|\cB^U_n|} -1\right|<\theta_1\;.
\end{align*}
\end{prop}
\begin{proof}
We start by fixing the constants needed in the proof. For $\vartheta:=\theta_1/8$ and $\ell=|U|$, we let $\gamma_0=\gamma_0(\vartheta,\ell)$ be the constant obtained from Proposition~\ref{lem:1}. Fix $\gamma\leq \gamma_0$.
For $\eta:=\frac{\theta_1}{4}$, we let $\epsilon_0=\epsilon_0(\gamma,\eta)$ be the constant obtained from Lemma~\ref{lem:bad_cells}. For $\epsilon:=\min\{\epsilon_0,1/\ell,\theta_1/8\}$, we let $k_0(\vartheta,\epsilon,\ell)$ be the maximum of the constants obtained from Lemma~\ref{lem:bad_cells} and Proposition~\ref{lem:1}.  Fix $k_*\geq k_0$. Let $\zeta$ be the minimum between the constant obtained from Lemma~\ref{lem:bad_cells} and $\theta_1/8$. Let $n$ be large enough with respect to all the previous parameters.

Now that $\epsilon$ and $k_*$ are fixed, we consider as before
the family  $\cU_{\epsilon}\subset \cU$ of unrooted trees of order at most $\lceil \epsilon^{-1}\rceil$ and the family $\cT_*\subset \cT$ of all rooted trees of order at most $k_*$. 
We also let $w$ and $K$, and the collection of boxes $\{[\beta_i]^w, 1\leq i\leq K\}$ be defined (relatively to the values of $\epsilon$ and $k_*$) as in Section~\ref{subsec:GoodAndBadBoxes}.
 We recall that these boxes satisfy~\eqref{eq:P1}, and using~\eqref{eq:Ppartition} we note that $\sum_{i=1}^K  |\mathcal{A}_{n, [\beta_i]^w_q}| = |\mathcal{A}_n|$.

We can write,
 \begin{align*}
\left| \frac{|\cB^U_n|}{|\cG_n|}- e^{-1/2}\frac{e^{-|U|}}{\Aut_u(U)}\right| 
 &\leq \left| \sum_{i=1}^K\frac{|\cB_{n,[\beta_i]^w}^U|}{|\cG_n|}- e^{-1/2}\frac{e^{-|U|}}{\Aut_u(U)}\right| +\epsilon\\
 &\leq \left| \sum_{i\notin \Good_{\gamma,\epsilon}}\frac{|\cB_{n,[\beta_i]^w}^U|}{|\cG_n|}+ \frac{1}{|\cG_n|}\sum_{i\in \Good_{\gamma,\epsilon}}|\cB_{n,[\beta_i]^w}^U| - e^{-1/2}\frac{e^{-|U|}}{\Aut_u(U)}\right|+\frac{\theta_1}{8}\;. 
  \end{align*}
By Proposition~\ref{lem:1}, for every $i\in \Good_{\gamma,\epsilon}$ and every $U\in \cU_{\leq \ell}$, we have
$$
\left||\cB_{n,[\beta_i]^w}^U|-\frac{e^{-|U|}}{\Aut_u(U)}\cdot |\cA_{n,[\beta_i]^w_q}|\right|\leq \frac{\theta_1}{8}\;.
$$
By Lemma~\ref{lem:bad_cells}, we have
 $$
 \sum_{i\notin \Good_{\gamma,\epsilon}}\frac{|\cB_{n,[\beta_i]^w}^U|}{|\cG_n|}\leq  \sum_{i\notin \Good_{\gamma,\epsilon}}\frac{|\cB_{n,[\beta_i]^w}|}{|\cB_n|}\leq  \eta = \frac{\theta_1}{4}\;.
 $$
Let $M$ be the number of boxes $[\beta_i]$ that are non-empty. Clearly, $M\leq |\cG_n|$. Therefore,
 \begin{align*}
&\left| \frac{|\cB^U_n|}{|\cG_n|}- e^{-1/2}\frac{e^{-|U|}}{\Aut_u(U)}\right|\\
&\;\;\; \leq \frac{\theta_1}{4} + \frac{\theta_1 M}{8|\cG_n|}+ \left| \frac{e^{-|U|}}{\Aut_u(U)|\cG_n|}\left(\sum_{i\in \Good_{\gamma,\epsilon}}|\cA_{n,[\beta_i]^w_q}|\right)- e^{-1/2}\frac{e^{-|U|}}{\Aut_u(U)}\right|  +\frac{\theta_1}{8}\\
&\;\;\; \leq \frac{\theta_1}{2}+\left| \frac{|\cA_n|}{|\cG_n|}\left(\frac{\sum_{i\in \Good_{\gamma,\epsilon}}|\cA_{n,[\beta_i]^w_q}|}{|\cA_n|}\right)- e^{-1/2}\right|\frac{e^{-|U|}}{\Aut_u(U)}
 \end{align*}
Again, by Lemma~\ref{lem:bad_cells} and using that $\sum_{i=1}^K  |\mathcal{A}_{n, [\beta_i]^w_q}| = |\mathcal{A}_n|$, we have
$$
\left|\frac{\sum_{i\in \Good_{\gamma,\epsilon}}|\cA_{n,[\beta_i]^w_q}|}{|\cA_n|} -1 \right| =\frac{\sum_{i\notin \Good_{\gamma,\epsilon}}|\cA_{n,[\beta_i]^w_q}|}{|\cA_n|}\leq \eta = \frac{\theta_1}{4}\;.
$$ 
Since $\cG$ is a $\zeta$-tight bridge-addable class, by definition, using Theorem~\ref{thm:conj} and provided that $n$ is large enough, we obtain
\begin{align*}
 (1-\zeta)e^{-1/2}\leq \frac{|\cA_n|}{|\cG_n|}\leq  (1+\zeta)e^{-1/2}\;.
\end{align*}
Since $\zeta\leq \theta_1/8$, we obtain 
 \begin{align*}
 \left| \frac{|\cB^U_n|}{|\cG_n|}- e^{-1/2}\frac{e^{-|U|}}{\Aut_u(U)}\right|
 & \leq \frac{\theta_1}{2}+\left(\left(1+\frac{\theta_1}{8}\right)\left(1+\frac{\theta_1}{4}\right) - 1\right)  e^{-1/2}\frac{e^{-|U|}}{\Aut_u(U)}\\
& \leq \frac{\theta_1}{2}+\frac{\theta_1}{2} \cdot e^{-1/2}\frac{e^{-|U|}}{\Aut_u(U)}\leq \theta_1\;.
 \end{align*}
 This concludes the proof of the first part of the proposition.
\medskip

For the second part, let us proceed by contradiction.  
Suppose that there exist $\theta$, $\delta$, $\ell$ and $U\in \mathcal{U}$, such that for every $\zeta$ there exists a $\zeta$-tight class $\cG$ and a large enough $n$ with
$$
\left| \frac{\sum_{\beta\in \Xi(\delta,\ell)}\left| \cB_{n,\beta}^{U} \right|}{|\cB^U_n|} -1\right|>\theta\;.
$$
or equivalently,
\begin{align}\label{eq:contra}
\frac{\sum_{\beta\notin \Xi(\delta, \ell)}\left| \cB_{n,\beta}^{U} \right|}{|\cB^U_n|}>\theta\;.
\end{align}

Note that by the first part of the proposition with $\theta_1$ small enough, we have that $\frac{|\cB^U_n|}{|\cG_n|}$ is arbitrarily close to $ e^{-1/2}\frac{e^{-|U|}}{\Aut_u(U)}$, for $\zeta$ small and $n$ large enough. Thus, there exists a uniform constant $c(U)>0$ such that $\frac{|\cB^{U}_n|}{|\cB_n|}\geq c(U)$, and~\eqref{eq:contra} is well-defined.

Let $\eta=\theta c(U)$ and let $\vartheta=\delta/2$. As in the first part of the proposition, we can choose $\gamma$, $\epsilon$, $k_*$, $\zeta$ and $n$, such that Lemma~\ref{lem:bad_cells} and Proposition~\ref{lem:2} can be applied. We skip the details of this setting.
We will again consider the set of boxes $\{[\beta_i]:\,1\leq i\leq K\}$ of $\cE_{k_*}$ fixed in Section~\ref{subsec:GoodAndBadBoxes}. For every $\alpha\in \cE_{k_*}$, we consider its canonical projection $\pi(\alpha)$ onto $\cE_\ell$ obtained by selecting the first $|\cE_\ell|$ coordinates of $\alpha$.

\begin{cla}
Let $\alpha\in [\beta_i]^w$, for some $i\in \Good(\gamma,\epsilon)$. Then $\pi(\alpha) \in \Xi(\delta, \ell)$.
\end{cla}
\begin{proof}[Proof of the Claim]
By Proposition~\ref{lem:2} and since $[\beta_i]$ is $(\gamma,\epsilon)$-good, for every $T\in \cT_{\leq \ell}$ we have 
$$
\left|\frac{\beta_i(T)}{n}-\frac{e^{-|T|}}{\Aut_r(T)}\right|< \vartheta\;.
$$
Since $\alpha\in [\beta_i]^w$, for every $T\in \cT_{\leq k_*}$, we have $|\beta_i(T)-\alpha(T)|\leq w$. The choice of $w$ does not depend on $n$, and thus, $|\frac{\beta_i(T)}{n}-\frac{\alpha(T)}{n}|\leq \frac{\delta}{3}$, if $n$ large enough.
Since $\ell\leq k_*$, for every $T\in \cT_{\leq \ell}$ we have 
$$
\left|\frac{\alpha(T)}{n}-\frac{e^{-|T|}}{\Aut_r(T)}\right|< \vartheta+\frac{\delta}{3}<\delta\;.
$$
We conclude that $\pi(\alpha)\in \Xi(\delta,\ell)$, which proves the claim.
\end{proof}

As a direct corollary of the claim, we get
\begin{align*}
\frac{\sum_{\beta\notin \Xi(\delta,\ell)}\left| \cB_{n,\beta}^{U} \right|}{|\cB^U_n|}
&\leq \frac{\sum_{i\notin \Good_{\gamma,\epsilon}}\left| \cB_{n,[\beta_i]^w}^{U} \right|}{|\cB^U_n|}\;.
\end{align*}
By Lemma~\ref{lem:bad_cells}, it follows that 
\begin{align*}
\frac{\sum_{\beta\notin \Xi(\delta,\ell)}\left| \cB_{n,\beta}^{U} \right|}{|\cB^U_n|}
&\leq \frac{|\cB_n|}{|\cB^U_n|}\cdot\frac{\sum_{i\notin \Good_{\gamma,\epsilon}}\left| \cB_{n,[\beta_i]^w}^{U} \right|}{|\cB_n|}
\leq \frac{|\cB_n|}{|\cB^U_n|}\cdot \eta
\leq \theta\;,
\end{align*}
where we have used $|\cB_{n,[\beta_i]^w}^{U}|\leq |\cB_{n,[\beta_i]^w}|$, giving a contradiction with~\eqref{eq:contra}.

\end{proof}

\subsection{Proof of Theorem~\ref{thm:main} for classes of forests}\label{ssc:k}

We now prove the main result of this section, Theorem~\ref{thm:main_forests}, that is equivalent to our main theorem for bridge-addable classes of forests.

Here we need to introduce the notion of \emph{removable edges}.
This notion will also be crucial in the next section in order to transfer the  obtained result from classes of forests to general graph classes.
We say that an edge $e$ in a graph $G\in\cG$ is \emph{removable} if the graph $G'=G\setminus e$ is in $\cG$.
For a subclass $\cH\subseteq \cG$ and a rooted tree $T\in \cT$, we define $p(\cH, T)$ to be the probability that given a uniform random graph $H\in\cH$, and a uniform random pendant copy of $T$ in $H$, the graph $H'$ obtained by deleting the edge that connects the pendant copy of $T$ to the rest of the graph belongs to $\cG$ (and not only to $\cH$). In other words, $p(\cH, T)$ is the average over all graphs in $\cH$ of the proportion of pendant copies of $T$ that are attached using a removable edge. This notion is inspired by bridge-alterable classes, for which $p(\cH,T)=1$, for every $\cH\subseteq \cG$ and every $T\in \cT$~\cite{ABMCR,KP}.
We do an slight abuse of notation by writing $p(G,T)$ for $p(\{G\},T)$, for each $G\in \cG$. Also, in the cases where $p(G,T)$ is not well-defined (that is, if 
$G$ has no pendant copy of $T$), we interpret the probability as~$1$.

Recall the definition of $\Xi(\delta,\ell)$ given in~\eqref{def:Xi}, and recall from Section~\ref{subsec:partition} that we use the notation $\{U_1,U_2,\dots,U_k\}$ to denote the forest formed by a \emph{multiset} of $k$ unrooted trees.
\begin{thm}\label{thm:main_forests}
For every $k\geq 1$, every $\theta_k$ and every $U_1,\dots, U_k \in \mathcal{U}$, there exists a $\zeta$ such that for every $\zeta$-tight class $\cG$ of forests and every  large enough $n$, one has
\begin{align}\label{eq:main_forest}
 \left| \frac{ \left| \cG_n^{k+1,\{U_1,\dots,U_k\}} \right|}{|\cG_n|}- e^{-1/2}\frac{e^{-\sum_{i=1}^k |U_i|}}{\Aut_u(U_1,\dots,U_k)}\right|< \theta_k\;.
\end{align}
Moreover, for every $k$, every $\ell$, every $\theta_k$, every $\delta$ and every $U_1,\dots, U_k \in \mathcal{U}$, there exists a $\zeta$ such that for every $\zeta$-tight class $\cG$ of forests and every large enough $n$, one has
\begin{align}\label{eq:main_forest2}
\left| \frac{\sum_{\beta\in \Xi(\delta,\ell)}\left| \cG_{n,\beta}^{k+1,\{U_1,\dots,U_k\}} \right|}{\left|\cG_{n}^{k+1,\{U_1,\dots,U_k\}}\right|} -1\right|<\theta_k\;.
\end{align}
\end{thm}

\noindent{\bf Note:} In the following proof and in the rest of the paper, we will often partition the set $\cG_n$ of graphs in $\cG$ having $n$ vertices into smaller graph classes. If $\mathcal{H}\subset \cG_n$ is a part in this partition, and $\zeta>0$ is a real number, we will say that $\mathcal{H}$ is $\zeta$-tight if
\begin{align*}
\Pr \left( H \mbox{ is connected} \right)\leq (1+\zeta)e^{-1/2}\;,
\end{align*}
where $H$ is uniform in $\mathcal{H}$. In other words, we slightly adapt the notion of $\zeta$-tightness for classes of graphs which from the context have a fixed number of vertices. 

\begin{proof}[Proof of Theorem~\ref{thm:main_forests}]
We will start by showing the first part of the theorem by induction on $k$, and then prove the second part, again by induction on $k$ and using the result obtained in the first part. The proof of the first part will also require the proof of an intermediate claim.
\medskip

{\noindent\bf Proof of the first part.}
Fix $U_1,\dots, U_k\in \cU$ and let $u$ be an upper bound on their sizes.

We are going to prove the first statement by induction. Observe that Proposition~\ref{prop:k1} proves the case $k=1$. Assume that the statement is true for $k-1$ and let us show it for $k$.

We consider the following total order on the subsets of $[n]$; for every $V_1,V_2\subseteq V$ we have $V_1<V_2$ if $|V_1|<|V_2|$ or $|V_1|=|V_2|$ and the elements of $V_1$ are smaller in the lexicographical order, than the ones in $V_2$.

Let $m(U_1,\dots,U_k)$ be the number of graphs isomorphic to $U_k$ among $U_1,\dots,U_k$. Observe that
\begin{align}\label{eq:autos}
\Aut_u(U_1,\dots,U_k)=m(U_1,\dots,U_k)\Aut_u(U_k)\Aut_u(U_1,\dots,U_{k-1})\;.
\end{align}
For every subset of vertices $W\subset V$, we use  $G[W]$ to denote the graph induced by $W$ in $G$. For every unlabeled graph $U$, the notation $G[W]\equiv U$, not only denotes graph isomorphism, but also that $W$ induces a maximal connected component in $G$.

Given disjoint sets $V_1,\dots,V_{k-1} \subset [n]$, consider the graph class
$$
\mathcal{H}(V_1,\dots,V_{k-1})=\{G[V\setminus \cup_{i=1}^{k-1} V_i]:\;G\in \cG_n ,\, G[V_1]\equiv U_{1},\dots, G[V_{k-1}]\equiv U_{{k-1}}\}\;.
$$
In what follows, given $U_1,\dots,U_{k-1},$ we will consider all the possible tuples of subsets $V_1,\dots,V_{k-1}$ underlying the small induced components of our graph. In order to avoid multiplicity problems and avoid considering the same tuple several times, we need to carefully define the set of such tuples in the case where several $U_i$'s are isomorphic to each other. To this end, we consider the set of $(k-1)$-tuples of disjoint subsets defined as follows,
\begin{align}\label{eq:defcV}
\cV=\{(V_1,\dots,V_{k-1}),\, V_i\subset [n] \text{ disjoint}; \text{ if } U_i \equiv U_j \text{ then }V_i<V_j\}\;.
\end{align}
We then write $\cH=\cup_{(V_1,\dots,V_{k-1})\in \cV} \cH(V_1,\dots,V_{k-1})$. 

Since $\cG$ is a bridge-addable class on the set of vertices $V$, then $\mathcal{H}(V_1,\dots,V_{k-1})$ (for every $(V_1,\dots, V_{k-1})\in \cV$) is also a bridge-addable class on the set of vertices $V\setminus \cup_{i=1}^{k-1} V_i$. It is worth to stress here that $|V\setminus \cup_{i=1}^{k-1} V_i|\geq n-(k-1)u$ is large enough (provided $n$ is large enough), and thus, our previous results can be applied to these classes of graphs. 

Consider the graphs in $\cG_n$ with $k+1$ components such that the $k$ smallest ones are isomorphic to $U_1,\dots, U_k$ and where one component isomorphic to $U_k$ is marked. By counting these graphs in two ways, for $n$ large enough, we have
\begin{align}\label{eq:goal2}
m(U_1,\dots,U_k)\left| \cG_n^{k+1,\{U_1,\dots,U_k\}} \right|=  \sum_{(V_1,\dots,V_{k-1})\in\cV} |\mathcal{H}^{2,U_k}(V_1,\dots,V_{k-1})|\,
\end{align}
Therefore,
\begin{align}\label{eq:goal3}
&\frac{\left| \cG_n^{k+1,\{U_1,\dots,U_k\}} \right|}{|\cG_n|}=\nonumber\\
&\;\;\;\;= \frac{1}{m(U_1,\dots,U_k)} \!\!  \sum_{(V_1,\dots,V_{k-1})\in \cV}\!\!\!\!\!\! \frac{|\mathcal{H}^{2,U_k}(V_1,\dots,V_{k-1})|}{|\mathcal{H}(V_1,\dots,V_{k-1})|} \cdot\frac{|\mathcal{H}(V_1,\dots,V_{k-1})|}{|\mathcal{H}^{(1)}(V_1,\dots,V_{k-1})|} \cdot\frac{|\mathcal{H}^{(1)}(V_1,\dots,V_{k-1})|}{|\cG_n|}\;.
\end{align}
Thus it suffices to estimate the three ratios in the sum above.

Let $\theta_1:=\frac{\theta_k}{8}$ and $\theta_{k-1}:=\frac{\theta_k}{8}$. 
Let $\zeta_1$ be the constant obtained from Proposition~\ref{prop:k1} with $\theta_1$ and $U=U_k$. Let $\zeta_2$ be the constant obtained by induction with $k-1$, $\theta_{k-1}$ and $U_1,\dots,U_{k-1}$. 
We set $\zeta_0:=\min\left\{\zeta_1,\zeta_2,\frac{\theta_k}{20},k^{-2}\right\}$.

Let us first show that most of graphs in $\cH$ are in classes $\mathcal{H}(V_1,\dots,V_{k-1})$ that are 
close to be tight. Let $\cV_0\subset \cV$ be the set of $(k-1)$-tuples such that $\cH(V_1,\dots,V_{k-1})$ satisfies
\begin{align}\label{eq:H_tight}
\Pr(H\in \cH(V_1,\dots,V_{k-1})\text{ connected})\geq (1+\zeta_0)e^{-1/2},
\end{align}
and let $\cH_0=\cup_{(V_1,\dots,V_{k-1})\in \cV_0} \cH(V_1,\dots,V_{k-1})$. 
\begin{cla}\label{cla:H}
There exists a $\zeta_3$ such that if $\cG$ is $\zeta_3$-tight and $n$ is large enough, we have
$$
|\cH_0|\leq \zeta_0|\cH|\;.
$$
\end{cla}
\begin{proof}[Proof of the Claim]
We first need to define the following subclasses that generalize $\mathcal{H}(V_1,\dots,V_{k-1})$.
For any $(k-1)$-tuple of trees $(W_1,W_2,\dots,W_{k-1})$, we define
$$
\cJ(W_1,\dots,W_{k-1};V_1,\dots,V_{k-1})=\{G[V\setminus \cup_{i=1}^{k-1} V_i]:\;G\in \cG_n,\, G[V_1]\equiv W_{1},\dots, G[V_{k-1}]\equiv W_{{k-1}}\}\;.
$$
Similarly as in~\eqref{eq:defcV} to avoid problems of multiplicity, we define the following subsets that generalize~$\cV$,
$$
\cV(W_1,\dots,W_{k-1})=\{(V_1,\dots,V_{k-1}),\, V_i\subset [n] \text{ disjoint}; \text{ if } W_i \equiv W_j \text{ then }V_i<V_j\}\;.
$$

We stress here that, by definition, for any non-empty class $\cJ(W_1,\dots,W_{k-1};V_1,\dots,V_{k-1})$ such that $\mathcal{H}(V_1,\dots,V_{k-1})$ is non-empty, we have $|W_i|\leq u$, for every $1\leq i\leq k-1$.
As before, we note that $\cJ(W_1,\dots,W_{k-1};V_1,\dots,V_{k-1})$ is bridge-addable. We will write $\cJ(W_1,\dots,W_{k-1})=\cup_{(V_1,\dots,V_{k-1})\in \cV(W_1,\dots,W_{k-1}) } \cJ(W_1,\dots,W_{k-1};V_1,\dots,V_{k-1})$ and we will write 
$$\cJ=\bigcup_{\{W_1,\dots,W_{k-1}\}} \cJ(W_1,\dots,W_{k-1}),$$
where the union is taken over multisets of trees $\{W_1,\dots,W_{k-1}\}$ 
and where for each multiset an arbitrary ordered tuple $(W_1,\dots,W_{k-1})$ is chosen.
Thus, $\cJ$ can be understood as the set of graphs in $\cG_n$ with at least $k$ components where exactly $k-1$ of the non-largest ones are marked.
In particular, we have:
\begin{align}\label{eq:Jbinomial}
|\cJ| &= \sum_{j\geq 0}\binom{k+j-1}{k-1} |\cG_n^{(k+j)}|\;.
\end{align}

Let $\eta=\zeta_0^3$ and $m$ such that $\sum_{\ell\geq m-k} \frac{1}{\ell!}\leq \eta$ and $m\geq k$. 
By Lemma~\ref{lem:tight} there exists a $\zeta_4$ such that if $\cG$ is $\zeta_4$-tight and $n$ is large enough, we have for every $i\leq m$
$$
\frac{|\cG^{(i)}_n|}{|\cG_n|} = 
\frac{\left(\frac{1}{2}\pm \zeta_0^3\right)^{i-1}}{(i-1)!}\;.
$$
Moreover, using the previous bound and~\eqref{eq:easyBound}, if $i> m$,
$$
\frac{|\cG^{(i)}_n|}{|\cG_n|} \leq 
\frac{\left(\frac{1}{2}+ \zeta_0^3\right)^{m}}{(i-1)!}\;.
$$
Therefore from~\eqref{eq:Jbinomial} we obtain
\begin{align}\label{eq:J}
|\cJ|  
		&= \frac{\left(\frac{1}{2}\pm \zeta_0^3\right)^{k-1}}{(k-1)!}\left(\sum_{j= 0}^{m-k}\frac{\left(\frac{1}{2}\pm \zeta_0^3\right)^{j}}{j!} \pm \left(\frac{1}{2}+ \zeta_0^3\right)^{m-k+1}\sum_{j> m-k}\frac{1}{j!} \right)|\cG_n|\nonumber \\
		&= \frac{\left(\frac{1}{2}\pm \zeta_0^3\right)^{k-1}}{(k-1)!}\left(e^{(1/2\pm \zeta_0^3)}\pm 2  \eta \right)|\cG_n|\nonumber \\
		&= \left(1\pm \frac{\zeta_0^2}{10}\right) e^{1/2}  |\cG_n^{(k)}|\;,
\end{align}
since $\zeta_0\leq k^{-2}$ and $\zeta_0$ is a small constant.

Now we set $\zeta_3:= \min\left\{ \frac{\zeta_0^2}{10(u^u)^k}, \zeta_4\right\}$. 
Fix $W_1,\dots,W_{k-1}$. Since $\cJ(W_1,\dots,W_{k-1})$ is a disjoint union of bridge-addable classes ($\cJ(W_1,\dots,W_{k-1};V_1,\dots,V_{k-1})$, for each $(V_1,\dots,V_{k-1})$) of graphs with $n-\sum_{j=1}^{k-1} |W_j|\geq n-(k-1)u$ vertices, if $n$ is large enough, by Theorem~\ref{thm:conj} applied to each class $\cJ(W_1,\dots,W_{k-1};V_1,\dots,V_{k-1})$, we have
\begin{align}\label{eq:local_J}
|\cJ(W_1,\dots,W_{k-1})|&\leq (1+\zeta_3) e^{1/2} |\cG_n^{k,\{W_1,\dots,W_{k-1}\}}|\nonumber\\
&\leq \left(1+\frac{\zeta_0^2}{10(u^u)^k}\right) e^{1/2} |\cG_n^{k,\{W_1,\dots,W_{k-1}\}}|\;.
\end{align}
Since there are at most $(u^u)^k$ multisets of unrooted trees $\{W_1,\dots,W_k\}$ of order at most $u$, from~\eqref{eq:J} and~\eqref{eq:local_J}, we have that for every $W_1,\dots,W_{k-1}$,
$$
|\cJ(W_1,\dots,W_{k-1})|\geq (1-\zeta_0^2/5) e^{1/2} |\cG_n^{k,\{W_1,\dots,W_{k-1}\}}|\;.
$$
This holds in particular for $\cH=\cJ(U_1,\dots,U_{k-1})$, implying
\begin{align}\label{eq:contradiction}
|\cG_n^{k,\{U_1,\dots,U_{k-1}\}}|  \leq (1+\zeta_0^2/4)  e^{-1/2} |\cH|\;,
\end{align}
since $\zeta_0$ is a small constant.

\noindent For the sake of contradiction assume now that $|\cH_0|\geq \zeta_0 |\cH|$.

Since $\cH\setminus \cH_0$ is a disjoint union of bridge-addable classes on $n-\sum_{j=1}^{k-1}|U_j|\geq n -(k-1)u$ vertices, provided that $n$ is large enough, Theorem~\ref{thm:conj} implies $\Pr(H\in \cH\setminus \cH_0\text{ connected})\geq (1-\zeta_3)e^{-1/2}$. 
Moreover, by definition of $\cH_0$, 
we have  $\Pr(H\in \cH_0\text{ connected})\geq (1+\zeta_0)e^{-1/2}$. We obtain
\begin{align*}
|\cG_n^{k,\{U_1,\dots,U_{k-1}\}}| &= \Pr(H\in \cH \text{ connected}) |\cH| \\
&= \Pr(H\in \cH\setminus \cH_0\text{ connected})|\cH\setminus \cH_0|+\Pr(H\in \cH_0\text{ connected})|\cH_0|\\
 &\geq \left((1-\zeta_3)|\cH\setminus \cH_0|+(1+\zeta_0)|\cH_0|\right)e^{-1/2} \\
 &\geq \left( 1+ \zeta_0^2 -\zeta_3+\zeta_0\zeta_3\right)e^{-1/2}|\cH| \\
 & \geq \left( 1+  \zeta_0^2/2\right)e^{-1/2}|\cH| \;,
\end{align*}
which gives a contradiction with~\eqref{eq:contradiction}. This concludes the proof of the claim.
\end{proof}

We now set $\zeta:=\min\{\zeta_0,\zeta_3\}$, where $\zeta_3$ is the one given by the previous claim.

Let $(V_1,\dots,V_{k-1})\in \cV\sm \cV_0$; that is, the class $\cH(V_1,\dots,V_{k-1})$ is $\zeta_0$-tight (and thus, also $\zeta_1$-tight). By Proposition~\ref{prop:k1} applied to the class $\mathcal{H}(V_1,\dots,V_{k-1})$, with the chosen $\theta_1$ and $U=U_k$,  and since the class is $\zeta_1$-tight and its elements have at least $n-\sum_{j=1}^{k-1} |V_j|\geq n-(k-1)u$ vertices, we have
\begin{align}\label{eq:a1}
\frac{|\mathcal{H}^{2,U_k}(V_1,\dots,V_{k-1})|}{|\mathcal{H}(V_1,\dots,V_{k-1})|}=e^{-1/2}\frac{e^{-|U_k|}}{\Aut_u(U_k)}\pm \frac{\theta_k}{8}\;.
\end{align}
Since $\cH(V_1,\dots,V_{k-1})$ is bridge-addable and since $(V_1,\dots,V_{k-1})\in\cV\sm \cV_0$, by Theorem~\ref{thm:conj} and 
by definition of $\cV_0$
\begin{align}\label{eq:a2}
\frac{|\mathcal{H}(V_1,\dots,V_{k-1})|}{|\mathcal{H}^{(1)}(V_1,\dots,V_{k-1})|} = e^{1/2}(1\pm \zeta_0)\;.
\end{align}
We proceed to bound the contribution of classes indexed by $\cV_0$. Using again the previous claim,
\begin{align}\label{eq:not_tight_V}
\sum_{(V_1,\dots,V_{k-1})\in \cV_0} |\mathcal{H}^{(1)}(V_1,\dots,V_{k-1})| 
&\leq |\cH_0| 
\leq \zeta_0 |\cH|\\
&\leq \zeta_0 (1-\zeta_0)^{-1} |\cH\sm \cH_0|\nonumber \\
&\leq 2\zeta_0 \sum_{(V_1,\dots,V_{k-1})\in \cV \sm \cV_0} |\mathcal{H}^{(1)}(V_1,\dots,V_{k-1})|\;, \nonumber
\end{align}
where the last inequality comes from~\eqref{eq:a2} and the fact that $\zeta_0$ is a small constant.
Therefore,
\begin{align*}
|\cG^{k, \{U_{1},\dots,U_{k-1}\}}_n| 
&= \sum_{(V_1,\dots,V_{k-1})\in \cV}  |\mathcal{H}^{(1)}(V_1,\dots,V_{k-1})| \\
& = (1\pm 2\zeta_0)\sum_{(V_1,\dots,V_{k-1})\in \cV\sm \cV_0} |\mathcal{H}^{(1)}(V_1,\dots,V_{k-1})| \;.
\end{align*}
Using the induction hypothesis for $k-1$, with the chosen $\theta_{k-1}$ and $U_1,\dots,U_{k-1}$, and since $\cG$ is $\zeta_2$-tight and its elements have at least $n-\sum_{j=1}^{k-1} |V_j|\geq n-(k-1)u$ vertices, it follows that
\begin{align}\label{eq:a3}
\sum_{(V_1,\dots,V_{k-1})\in\cV\sm \cV_0} \frac{|\mathcal{H}^{(1)}(V_1,\dots,V_{k-1})|}{|\cG_n|}
&= (1\pm 2\zeta_0)^{-1}\frac{ \left| \cG_n^{k,\{U_1,\dots,U_{k-1}\}} \right|}{|\cG_n|}\nonumber\\
&= (1\pm 2\zeta_0)^{-1}\left(e^{-1/2}\frac{e^{-\sum_{i=1}^{k-1}|U_{i}|}}{\Aut_u(U_{1},\dots,U_{k-1})}\pm \frac{\theta_k}{8}\right)\nonumber\\
&= e^{-1/2}\frac{e^{-\sum_{i=1}^{k-1}|U_{i}|}}{\Aut_u(U_{1},\dots,U_{k-1})}\pm \frac{\theta_k}{4}\;.
\end{align} 
We are now ready to estimate~\eqref{eq:goal3}. We rewrite~\eqref{eq:goal3} as 
$$
\frac{\left| \cG_n^{k+1,\{U_1,\dots,U_k\}} \right|}{|\cG_n|} = \frac{1}{m(U_1,\dots,U_k)}(\Sigma_{\cV_0} + \Sigma_{\cV\sm\cV_0})\;.
$$
where $\Sigma_{\cV_0}$ and $\Sigma_{\cV\sm\cV_0}$  are the contribution to the sum of the elements indexed by $\cV_0$ and by $\cV\sm \cV_0$, respectively. 

In order to estimate $\Sigma_{\cV_0}$, we note that $|\cH(V_1,\dots,V_{k-1})|\leq e|\cH^{(1)}(V_1,\dots,V_{k-1})|$, since the class $\cH(V_1,\dots,V_{k-1})$ is bridge-addable and using Theorem~2.5 in~\cite{MCSWplanar}. Using~\eqref{eq:not_tight_V}, we obtain
$$
\Sigma_{\cV_0}\leq \frac{e\zeta_0 |\cH|}{|\cG_n|}\leq 3 \zeta_0< \frac{\theta_k}{2}\;.
$$
To estimate $\Sigma_{\cV\sm \cV_0}$, we use~\eqref{eq:a1},~\eqref{eq:a2} and~\eqref{eq:a3}, to obtain that
\begin{align*}
\Sigma_{\cV\sm\cV_0}
&= e^{-1/2}\frac{e^{-\sum_{i=1}^k |U_i|}}{\Aut_u(U_{k})\Aut_u(U_{1},\dots,U_{k-1})}  \pm \frac{\theta_k}{2} \\
\end{align*}
Using the previous two estimates and~\eqref{eq:autos}, we get 
\begin{align*}
\frac{ \left| \cG_n^{k+1,\{U_1,\dots,U_k\}} \right|}{|\cG_n|}
&= \frac{1}{m(U_1,\dots,U_k)}\cdot e^{-1/2}\frac{e^{-\sum_{i=1}^k |U_i|}}{\Aut_u(U_{k})\Aut_u(U_{1},\dots,U_{k-1})}  \pm \theta_k \\
&=  e^{-1/2}\frac{e^{-\sum_{i=1}^k |U_i|}}{\Aut_u(U_{1},\dots,U_{k})}  \pm \theta_k\;.
\end{align*}
This concludes the proof of the first part.
\medskip

{\noindent\bf Proof of the second part.}
We will prove the second part of the theorem, using the first part of it and by induction on $k$. For $k=1$, the statement we want to prove is directly given by Proposition~\ref{prop:k1}. Assume now that the statement is true for $k-1$. 

Set $\hat\theta_k:=e^{-1/2}\frac{e^{-ku}}{(ku)!}\theta_k $.
By the induction hypothesis, for $\ell$, $\theta_{k-1}:= \frac{\hat \theta_k}{8}$, $\delta_{k-1}:=2\delta$ and $U_1,\dots, U_{k-1}$, there exists a $\zeta_{k-1}$ such that if $n$ is large enough, we have
\begin{align}\label{eq:whatwehave}
\frac{\sum_{\beta\notin \Xi(\delta_{k-1},\ell)}\left| \cG_{n,\beta}^{k,\{U_1,\dots,U_{k-1}\}} \right|}{\left|\cG_{n}^{k,\{U_1,\dots,U_{k-1}\}} \right|} <\frac{\hat \theta_k}{8}\;.
\end{align}
Since the first part of the theorem for $k$ is already proved, we use it to estimate the ratio between $\cG_{n}^{k+1,\{U_1,\dots,U_{k}\}}$ and   $\cG_{n}^{k,\{U_1,\dots,U_{k-1}\}}$. 
For the first one we use the first part of the theorem for $k$ with $\theta_k:=\frac{\hat \theta_k}{8}$ and $U_1,\dots,U_{k}$ and the corresponding $\zeta_{k}'$.
For the second one we use, as before, the first part of the theorem for $k-1$ with  $\theta_{k-1}$ and $U_1,\dots,U_{k-1}$ and the corresponding $\zeta_{k-1}$.
Set $\zeta:=\min\{\zeta_{k-1},\zeta_{k}'\}$ and let $n$ be large enough. 

Using~\eqref{eq:autos}, it follows that
\begin{align}\label{eq:ratio}
\frac{|\cG_{n}^{k+1,\{U_1,\dots,U_{k}\}}|}{|\cG_{n}^{k,\{U_1,\dots,U_{k-1}\}}|}&=  \frac{e^{-1/2}\frac{e^{-\sum_{i=1}^k |U_i|}}{\Aut_u(U_{1},\dots,U_{k})}  \pm \hat\theta_k/8}{e^{-1/2}\frac{e^{-\sum_{i=1}^{k-1} |U_i|}}{\Aut_u(U_{1},\dots,U_{k-1})}  \pm \hat\theta_k/8}\nonumber\\
&= \frac{e^{-|U_k|}}{m(U_1,\dots,U_{k-1})\Aut_u(U_k)}\left(1\pm \frac{ \theta_k}{3}\right)\;.
\end{align}
Let $T_1,\dots,T_s$ be all the possible rooted versions of the unrooted tree $U_k$. Observe that $|T_i|=|U_k|$ and that
\begin{align}\label{eq:more_autos}
\sum_{i=1}^s \frac{1}{\Aut_r(T_i)} = \frac{|U_k|}{\Aut_u(U_k)}\;.
\end{align}

Recall the definition of $p(\cH,T)$ given at the beginning of Section~\ref{ssc:k}.  
We perform an exact double-counting argument between the graphs in $\cG_n^{k,\{U_1,\dots,U_{k-1}\}}$ and in $\cG_n^{k+1,\{U_1,\dots,U_k\}}$ using $p(G,T_i)$ with $G\in \cG_n^{k,\{U_1,\dots,U_{k-1}\}}$, similar to the one used in~Section~\ref{subsec:components}. In one direction,  for any such graph $G$, we have exactly $\sum_{i=1}^s \alpha^G(T_i) p(G,T_i)$ ways to construct a graph $G'\in\cG_n^{k+1,\{U_1,\dots,U_{k}\}}$ by removing an edge. In the other direction, there are exactly $m(U_1,\dots,U_k)|U_k|(n-\sum_{i=1}^{k}|U_j|)$ ways to obtain a graph in $\cG_n^{k,\{U_1,\dots,U_{k-1}\}}$ from one in $\cG_n^{k+1,\{U_1,\dots,U_{k}\}}$ by adding an edge. 
Therefore, we have
\begin{align}\label{eq:DC}
\sum_{G\in \cG_n^{k,\{U_1,\dots,U_{k-1}\}}} \sum_{i=1}^s \alpha^G(T_i) p(G,T_i) 
=m(U_1,\dots,U_k)|U_k|\left(n-\sum_{i=1}^{k}|U_j|\right)|\cG_n^{k+1,\{U_1,\dots,U_{k}\}}| 
\end{align}

Using~\eqref{eq:ratio} and~\eqref{eq:more_autos}, it follows that
\begin{align*}
\frac{\sum_{G\in \cG_n^{k,\{U_1,\dots,U_{k-1}\}}} \sum_{i=1}^s \alpha^G(T_i) p(G,T_i)}{n |\cG_n^{k,\{U_1,\dots,U_{k-1}\}}|} 
&=\frac{m(U_1,\dots,U_k)|U_k|\left(n-\sum_{i=1}^{k}|U_j|\right)|\cG_n^{k+1,\{U_1,\dots,U_{k}\}}| }{n |\cG_n^{k,\{U_1,\dots,U_{k-1}\}}|}\\
& = \frac{n-\sum_{i=1}^{k}|U_j|}{n}\cdot\frac{|U_k| e^{-|U_k|}}{\Aut_u(U_k)}\left(1 \pm \frac{ \theta_k}{3}\right)\\
& = \sum_{i=1}^s \frac{ e^{-|T_i|}}{\Aut_r(T_i)}\left(1\pm\frac{ \theta_k}{2}\right)\;,
\end{align*}
provided that $n$ is large enough.

Since for every $G\in \cG_n$, $\sum_{i=1}^s \alpha^G(T_i) p(G,T_i)\leq n$, it follows that
\begin{align}\label{eq:interesting}
&\frac{\sum_{\beta\in \Xi(\delta_{k-1},\ell)} \left(\sum_{i=1}^s \beta(T_i) p(\cG_{n,\beta}^{k,\{U_1,\dots,U_{k-1}\}},T_i)\right)\cdot|\cG_{n,\beta}^{k,\{U_1,\dots,U_{k-1}\}}| }{ n |\cG_n^{k,\{U_1,\dots,U_{k-1}\}}|} \nonumber\\
&\;\;\;\;\;\;\;\; = \frac{\sum_{G\in \cG_n^{k,\{U_1,\dots,U_{k-1}\}}} \sum_{i=1}^s \alpha^G(T_i) p(G,T_i)}{ n |\cG_n^{k,\{U_1,\dots,U_{k-1}\}}| } \pm \frac{\sum_{\beta\notin \Xi(\delta_{k-1},\ell)} |\cG_{n,\beta}^{k,\{U_1,\dots,U_{k-1}\}}|}{{|\cG_n^{k,\{U_1,\dots,U_{k-1}\}}|} } \nonumber\\
&\;\;\;\;\;\;\;\;=  \sum_{i=1}^s \frac{ e^{-|T_i|}}{\Aut_r(T_i)}\left(1 \pm \frac{5 \theta_k}{8}\right)\;.
\end{align}
If $G'$ is obtained from $G$ by removing an edge that creates a component isomorphic to $U_k$, then $|\alpha^G(T)-\alpha^{G'}(T)|\leq |U_k|\leq u$ for every $T\in \cT$.
Therefore, if $G\in \cG_{n,\alpha}^{k,\{U_1,\dots,U_{k-1}\}}$ for some $\beta\in \Xi(\delta_{k-1},\ell)$, then $G'\in \cG_{n}^{k+1,\{U_1,\dots,U_{k}\}}$ is such that $\alpha^{G'}\in \Xi(\delta,\ell)$ (recall that $\delta_{k-1}=\delta/2$), provided that $n$ is large enough. We thus obtain a local version of~\eqref{eq:DC}
\begin{align*}
&\sum_{\beta\in \Xi(\delta_{k-1},\ell)}   \left(\sum_{i=1}^s \beta(T_i) p(\cG_{n,\beta}^{k,\{U_1,\dots,U_{k-1}\}},T_i)  \right)
\cdot|\cG_{n,\alpha}^{k,\{U_1,\dots,U_{k-1}\}}|\\
&\;\;\;\;\leq m(U_1,\dots,U_k)|U_k|\left(n-\sum_{i=1}^{k}|U_j|\right)\sum_{\beta\in \Xi(\delta,\ell)} |\cG_{n,\beta}^{k+1,\{U_1,\dots,U_{k}\}}|
\end{align*}

Using~\eqref{eq:interesting}, the last inequality and~\eqref{eq:ratio}, it follows that
\begin{align*}
\sum_{i=1}^s &\frac{ e^{-|T_i|}}{\Aut_r(T_i)}\left(1 - \frac{5 \theta_k}{8}\right) \\
& \leq \frac{1}{n|\cG_{n}^{k,\{U_1,\dots,U_{k-1}\}}|}\sum_{\beta\in \Xi(\delta_{k-1},\ell)} \left(\sum_{i=1}^s \beta(T_i) p(\cG_{n,\beta}^{k,\{U_1,\dots,U_{k-1}\}},T_i) \right) \cdot|\cG_{n,\beta}^{k,\{U_1,\dots,U_{k-1}\}}|\\
&\leq \frac{e^{-|U_k|}}{\Aut_u(U_k)|\cG_{n}^{k+1,\{U_1,\dots,U_{k}\}}|} \sum_{\beta\in \Xi(\delta,\ell)} \frac{(n-\sum_{i=1}^{k}|U_j|)|U_k|}{n} |\cG_{n,\beta}^{k+1,\{U_1,\dots,U_{k}\}}| \left(1+\frac{ \theta_k}{3}\right)\\
&\leq \frac{|U_k| e^{-|U_k|}}{\Aut_u(U_k)} \cdot\frac{\sum_{\beta\in \Xi(\delta,\ell)} |\cG_{n,\beta}^{k+1,\{U_1,\dots,U_{k}\}}|}{|\cG_{n}^{k+1,\{U_1,\dots,U_{k}\}}|}\left(1+\frac{ \theta_k}{3}\right)\\
&= \sum_{i=1}^s \frac{ e^{-|T_i|}}{\Aut_r(T_i)} \cdot\frac{\sum_{\beta\in \Xi(\delta,\ell)} |\cG_{n,\beta}^{k+1,\{U_1,\dots,U_{k}\}}|}{|\cG_{n}^{k+1,\{U_1,\dots,U_{k}\}}|}\left(1+\frac{ \theta_k}{3}\right)\;,
\end{align*}
where we used~\eqref{eq:more_autos} for the last equality. We conclude,
\begin{align*}
\frac{\sum_{\beta\in \Xi(\delta,\ell)} |\cG_{n,\beta}^{k+1,\{U_1,\dots,U_{k}\}}|}{|\cG_{n}^{k+1,\{U_1,\dots,U_{k}\}}|} &\geq 1-\theta_k\;,
\end{align*}
which finishes the proof of the theorem.
\label{pageTbis}
\end{proof}

\section{From classes of forests to classes of graphs}\label{sec:transfer}

In this section we extend the results of the previous section (where we obtained Theorem~\ref{thm:main} for classes of forests) to general bridge-addable classes, concluding the proof of Theorem~\ref{thm:main}.
In~\ref{subsec:removable} we prove that graphs in $\zeta$-tight bridge-addable classes tend to have many removable edges, and in~\ref{subsec:conclude} we use this property and the results of Section~\ref{sec:trees} to conclude the proof of Theorem~\ref{thm:main}. Finally in~\ref{subsec:BS} we give the proof of Corollary~\ref{cor:BS}.

\subsection{Removable edges in tight bridge-addable classes of graphs}
\label{subsec:removable}

A \emph{$2$-block} of a graph $G$ is a maximal $2$-edge-connected graph (we assume that the graph composed by a single vertex is also $2$-edge-connected). 
Every graph admits a unique decomposition into $2$-blocks, joined by edges in a tree-like fashion. 

For a graph class $\cG_n$, we can consider the coarsest partition 
\begin{align}\label{eq:subclasses}
\cG_n = \biguplus_i \cH^{[i]}_n
\end{align}
into subclasses $\cH^{[1]}_n, \cH^{[2]}_n,\dots$ such that every two graphs in the same subclass have the same $2$-blocks. By construction, if $\cG_n$ is bridge-addable, then every subclass $\cH^{[i]}_n$ is also bridge-addable.

For each such subclass $\cH$, we assume that we have chosen, arbitrarily and once and for all, a spanning tree for each $2$-block of the graphs in $\cH$. We denote by $\cF_{\cH}$ the class of forests obtained by replacing each $2$-block with the corresponding spanning tree in each graph in $\cH$. This is well-defined, since, by construction, graphs in the same subclass have the same $2$-blocks. Moreover, the class $\cF_\cH$ is also bridge-addable and the component structure (number and size) of each graph $H\in \cH$ is preserved in the corresponding forest $F_H\in \cF_{\cH}$. This construction was introduced in~\cite{BBG}, to which we refer for more details.
\medskip

We start with the following useful lemma that says, loosely speaking, that most graphs in a $\zeta$-tight belong to subclasses $\cH^{[i]}_n$ that are themselves close to be tight.
\begin{lemma}\label{lemma:tightsubclasses}
For every $\zeta_0>0$ there exists $\zeta>0$ such that if $n$ is large enough, for any bridge-addable class $\cG$ that is $\zeta$-tight, the following is true. Let $\cH^{[1]}_n, \cH^{[2]}_n,\dots$ be the partition of $\cG_n$ in bridge-addable subclasses defined above and let $S_n(\zeta_0)$ be the set of values $i$ such that:
\begin{align}\label{eq:Hi_tight}
\Pr(H_n\in \cH^{[i]}_n\text{ connected})\leq (1+\zeta_0)e^{-1/2},
\end{align}
where $H_n\in \cH^{[i]}_n$ denotes a uniform random graph in $\cH^{[i]}_n$.
Then we have
\begin{align}\label{eq:Hi_tight2}
	\left|\bigcup_{i\in S_n(\zeta_0)} \cH^{[i]}_n\right|\geq (1-\zeta_0) |\cG_n|\;.
\end{align}
\end{lemma}
\begin{proof}
	The proof is direct by an averaging argument in a similar way as in the claim inside the proof of Theorem~\ref{thm:main_forests}.
\end{proof}

The next theorem states that $\zeta$-tight bridge-addable classes of graphs (not only forests) are, in fact, close to be bridge-alterable. In what follows, we say that a vertex $v$ in $G_n$ is \emph{connected to the bulk of $G_n$ through a cut-edge}, if there is a cut-edge $e$ incident to $v$ such after removing $e$, the newly created component not containing $v$ has size at least $3n/4$. Note that for each $v$ there is at most one edge $e$ with this property. The connected component containing $v$ after removing $e$ is called a \emph{pendant graph}.
The edge $e$ can {\it a priori} be removable or not, and if it is we say that $v$ is \emph{connected to the bulk of $G_n$ through a removable cut-edge}.

\begin{lemma}\label{lem:remove_cuts}
For every $\theta$, there exist a $\zeta$ and an $\ell$ such that
provided that $n$ is large enough, for every $\zeta$-tight bridge-addable class $\cG$,
 we have that if $G_n$ is a graph chosen uniformly at random in $\cG_n$, and $V_n$ is a vertex chosen uniformly at random in $G_n$, the following holds with probability at least $1-\theta$:
$V_n$ is connected to the bulk of $G_n$ through a removable cut-edge and the corresponding pendant graph has order at most $\ell$.
\end{lemma}

\begin{proof}
We first prove the lemma for bridge-addable classes of forests and then we transfer it to general bridge-addable classes of graphs.

Assume that $\cG_n$ is composed by forests. We first show that there exists an $\ell$ such that if $G_n$ is a graph chosen uniformly at random from $\cG_n$, then with probability at least $(1-\theta/4)$ we have that $p(\cG_n,T)\geq 1-\theta/4$ for every $T\in \cT_{\leq \ell}$. Then we will prove that with probability at least $1-\theta$, most of the pendant trees in $G_n$ have size at most $\ell$.

From Lemma~\ref{lemma:evalGF}, we can choose  $\ell$  large enough such that
\begin{align*}
\sum_{T\in\cT_{\leq \ell}} \frac{e^{-|T|}}{\Aut_r(T)}&\geq 1-\frac{\theta}{10}\;, \\
e^{-1/2} \sum_{k=0}^{\ell} \sum_{\{U_1,\dots,U_k\}\in \cU_{\leq \ell}} \frac{e^{-\sum_{i=1}^k |U_i|}}{\Aut_u(U_1,\dots,U_k)} &\geq 1-\frac{\theta}{10}\;.
\end{align*}

Let $T\in \cT_{\leq \ell}$ be a given rooted tree, we will show that $p(\cG_n, T)\geq 1-\theta/4$.
Let $\lambda$ be the size of the equivalence class of the root of $T$ (that is the number of vertices where $T$ can be re-rooted giving rise to another copy of $T$).  For every $k\leq \ell$ and every $U_1,\dots,U_k$ of order at most $\ell$ such that $U_k$ is the unrooted version of $T$, we will write the ratio between $|\cG_n^{k+1,\{U_1,\dots,U_k\}}|$ and $|\cG_n|$ in two ways.
We select $\zeta$ small enough and $n$ large enough, such that we can apply Theorem~\ref{thm:main_forests} for every $k\leq \ell$, for $\theta_k=\tilde{\theta}$ (to be fixed later) and for every $U_1,\dots,U_k$ of size at most $\ell$. If $\cG$ is $\zeta$-tight, we obtain
\begin{align*}
  \frac{|\cG_{n}^{k+1,\{U_1,\dots,U_{k}\}}|}{|\cG_n|}
& = e^{-1/2}\frac{e^{-\sum_{i=1}^k |U_i|}}{\Aut_u(U_1,\dots,U_k)}\pm \tilde{\theta}\;.
\end{align*}

As before, we perform an exact local double-counting argument with the difference that now we only count those graphs $G'\in \cG_{n}^{k+1,\{U_1,\dots,U_{k}\}}$ that can be obtained from $G\in \cG_{n}^{k,\{U_1,\dots,U_{k-1}\}}$ by removing an edge from where a copy of $T$ is pending. This can only be done if $U_k$ is the unrooted version of $T$ and if the edge that connects $T$ to the rest of $G$ is removable. Moreover, if $G'$ is obtained from $G$ in such a way, for every $T_0\in \cT_{\leq \ell}$ we have $|\alpha^G(T_0)-\alpha^{G'}(T_0)|\leq |T|\leq \frac{\tilde{\theta}n}{2}$. In one direction, given a graph $G\in \cG_{n}^{k,\{U_1,\dots,U_{k-1}\}}$ there are exactly $p(G,T)\alpha^G(T)$ many such ways to obtain a graph in $\cG_{n}^{k+1,\{U_1,\dots,U_{k}\}}$, and in the other one, exactly $\lambda m(U_1,\dots,U_k) (n-\sum_{j=1}^{k} |U_j|)$ many ones.
Applying Theorem~\ref{thm:main_forests} twice with $\theta_k=\tilde{\theta}$ and $\delta=\tilde{\theta}/2$, if $\zeta$ is small enough and $n$ is large enough, then if $\cG$ is $\zeta$-tight, we obtain
\begin{align*}
&\frac{  |\cG_{n}^{k+1,\{U_1,\dots,U_{k}\}}|  }{|\cG_n|}
  \leq \frac{1}{|\cG_n|}\sum_{\beta\in \Xi(\delta,\ell)}  |\cG_{n,\beta}^{k+1,\{U_1,\dots,U_{k}\}}|(1+ \tilde{\theta})\\
 &\;\;\;\leq  \frac{1}{|\cG_n| m(U_1,\dots,U_k)}\sum_{\beta\in \Xi(\tilde{\theta}, \ell)}  \sum_{G\in \cG_{n,\beta}^{k,\{U_1,\dots,U_{k-1}\}}} \frac{p(G,T) \alpha^G(T)}{(n-\sum_{j=1}^{k-1} |U_j|) \lambda}(1+ \tilde{\theta} )\\
 &\;\;\;=  \frac{1}{|\cG_n| m(U_1,\dots,U_k)}\sum_{\beta\in \Xi(\tilde{\theta}, \ell)}  |\cG_{n,\beta}^{k,\{U_1,\dots,U_{k-1}\}}|\frac{p(\cG_{n,\beta}^{k,\{U_1,\dots,U_{k-1}\}},T) \beta(T)}{(n-\sum_{j=1}^{k-1} |U_j|) \lambda}(1+ \tilde{\theta} )\\
&\;\;\;\leq \frac{1}{ m(U_1,\dots,U_k)}\cdot\frac{ |\cG_{n}^{k,\{U_1,\dots,U_{k-1}\}}|}{|\cG_n|}\cdot\frac{p(\cG'_n,T)}{\lambda} \left(\frac{ e^{-|T|}}{\Aut_r(T)}+\tilde{\theta}\right) (1+ \tilde{\theta})\\
 &\;\;\;\leq \frac{1}{  m(U_1,\dots,U_k)}\left(e^{-1/2}\frac{e^{-\sum_{i=1}^{k-1} |U_i|}}{\Aut_u(U_1,\dots,U_{k-1})}+\tilde{\theta}\right)  p(\cG'_n,T)\frac{ e^{-|U_k|}}{\Aut_u(U_k)} (1+ 3\tilde{\theta})\\
 &\;\;\;\leq e^{-1/2}\frac{e^{-\sum_{i=1}^{k} |U_i|}}{\Aut_u(U_1,\dots,U_{k})}  \cdot p(\cG'_n,T)(1 + 5\tilde{\theta})\;.
\end{align*}
where $\cG'_n$ is the class formed by the union of $\cG_{n,\beta}^{k,\{U_1,\dots,U_{k-1}\}}$ for $\beta\in \Xi(\delta,\ell)$. In the previous inequalities we have used that  $\Aut_u(U_k)=\lambda\Aut_r(T)$  and~\eqref{eq:more_autos}.

Combining these two expressions and  since $|\cG'_n| \geq (1-\tilde{\theta}) |\cG_{n}^{k,\{U_1,\dots,U_{k-1}\}}|$ (by Theorem~\ref{thm:main_forests}), we obtain that for every rooted tree $T\in\cT_{\leq \ell}$, 
\begin{align}\label{eq:removable:local}
p(\cG_{n}^{k,\{U_1,\dots,U_{k-1}\}},T) \geq 1- 8\tilde{\theta}\;.
\end{align}
Now we set  $\tilde{\theta}:=\theta {\ell}^{-(\ell^2+1)}/10$.  Applying Theorem~\ref{thm:main_forests} for every $k\leq \ell$, $\theta_k=\tilde{\theta}$ and $U_1,\dots,U_k$, and using the definition of $\ell$
\begin{align}\label{eq:most in t*}
\sum_{k=0}^{\ell} \sum_{U_1,\dots,U_k\in \cU_{\leq \ell}} \frac{|\cG_n^{k+1,\{U_1,\dots,U_k\}}|}{|\cG_n|} &=e^{-1/2} \sum_{k=0}^{\ell} \sum_{U_1,\dots,U_k\in \cU_{\leq \ell}} \frac{e^{-\sum_{i=1}^k |U_i|}}{\Aut_u(U_1,\dots,U_k)} - \tilde{\theta} \ell (\ell^{\ell})^{\ell}  \nonumber\\
&\geq 1-\frac{\theta}{5}\;,
\end{align}
By averaging~\eqref{eq:removable:local} over all $k$ and $U_1,\dots,U_{k-1}$ and using  the last equation, for every $T\in \cT_{\leq \ell}$, we obtain
\begin{align}\label{eq:removable}
p(\cG_n,T) \geq 1- \theta/4\;,
\end{align}
which proves the first part.
\vspace{0.2cm}

Let us now show that there are many removable edges that isolate a tree of size at most $\ell$. Choose $G_n$ uniformly at random from $\cG_n$ and then choose $V_n$ uniformly at random from $V(G_n)$.  Let  $A_1$ be the event that  $V_n$ is connected to the bulk of $G_n$ through a removable cut-edge and let $A_2$ be the event that the pendant tree rooted at $V_n$ has order at most $\ell$. We want to show that $\Pr(A_1\cap A_2)\geq 1-\theta$.

Again, by applying Theorem~\ref{thm:main_forests} for every $k\leq \ell$, $\theta_k=\tilde{\theta}$ and $U_1,\dots,U_k$, and using~\eqref{eq:most in t*}, we obtain
\begin{align*}
\sum_{\beta\in\Xi(\delta,\ell)}|\cG_{n,\beta}|&\geq \sum_{k=0}^{\ell} \sum_{U_1,\dots,U_k\in \cU_{\leq \ell}}\sum_{\beta\in\Xi(\delta,\ell)} |\cG_{n,\beta}^{k+1,\{U_1,\dots,U_k\}}| \\
&\geq \sum_{k=0}^{\ell} \sum_{U_1,\dots,U_k\in \cU_{\leq \ell}} |\cG_{n}^{k+1,\{U_1,\dots,U_k\}}| - \tilde{\theta} \ell(\ell^{\ell})^{\ell}\\
&\geq (1-\theta/4)  |\cG_{n}|\;.
\end{align*}
Moreover, for every $\beta\in\Xi(\delta,\ell)$ and by our choice of $\ell$, we have that $\sum_{T\in\cT_{\leq \ell}} \frac{\beta(T)}{n} \geq \sum_{T\in\cT_{\leq \ell}}  \frac{e^{|T|}}{\Aut_r(T)} - \delta \ell^{\ell}\geq 1-\theta/5$. It follows that $\Pr(A_2)\geq 1-\theta/2$.

Moreover, the probability of $A_1$ given $A_2$ can be written as a convex combination of $p(\cG_n,T)$ with $T\in \cT_{\leq \ell}$. Therefore, by~\eqref{eq:removable}, we have that $\Pr(A_1\mid A_2)\geq 1-\theta/4$.

We conclude that
\begin{align}\label{eq:first_conclusion}
\Pr( A_1\cap A_2)= 1-\Pr( \overline{A_1}\cup \overline{A_2})\geq 1-(\Pr(\overline{A_2}) +\Pr(\overline{A_1}\mid A_2)) \geq 1-3\theta/4\;,
\end{align}
which concludes the proof of the theorem when all graphs in $\mathcal{G}$ are forests.
\medskip

In order to extend the result to general classes of graphs, we use the approach introduced in~\cite{BBG}. For such a purpose, let $\cG_n$ be a general class of graphs.

Let $\cH^{[1]}_n, \cH^{[2]}_n,\dots$ be the partition of $\cG_n$ into subclasses defined at the beginning of this section. Given $\zeta_0$ (to be fixed later), we let $S_n=S_n(\zeta_0)$ be the set of indices given by Lemma~\ref{lemma:tightsubclasses}, and we fix an index $i\in S_n$. We let  $\cH:=\cH^{[i]}_n$ be the corresponding subclass of $\cG_n$ and we let $\cF_{\cH}$ be the corresponding class of forests. We observe that $\cF_{\cH}$ is $\zeta_0$-tight and bridge-addable.

Since Lemma~\ref{lem:remove_cuts} holds for classes of forests, we can apply it to $\cF_\cH$. Note that if a cut-edge is removable for a forest $F_H\in\cF_\cH$, then the edge does not belong to any of the $2$-blocks of the corresponding graph $H\in \cH$. This implies that this cut-edge is also removable for $H\in \cH$. Moreover, if its removal in $F_H$ results in a tree of size at most $\ell$, then its removal in $H$ results in a \emph{graph} of size at most $\ell$. Therefore, the result obtained in~\eqref{eq:first_conclusion} for $\cF_\cH$ naturally transfers to the class $\cH$, provided we change ``trees'' by ``graphs'' in what results after deleting a removable edge.

Moreover, if we choose $\zeta_0$ small enough with respect to $\theta$, then there exists $\zeta$ such that if $
\cG$ is $\zeta$-tight and $n$ is large enough, by~\eqref{eq:Hi_tight2}, at least $(1-\theta/4)|\cG_n|$ graphs in $\cG_n$ are in subclasses $\cH^{[i]}_n$ with $i\in S_n$, concluding that the lemma also holds for general classes of graphs $\cG_n$. 
\end{proof}

 A direct consequence of the previous lemma is that, for most of the graphs in $\cG_n$, most of their vertices are in $2$-blocks of size $1$.
Indeed, if a graph with $n$ vertices has $(1-\theta')n$ cut-edges for some $\theta'>0$, it has at least $(1-2\theta')n$ $2$-blocks of size one by an easy counting argument.

Our next goal is to use this observation to show that not only the pendant graphs obtained when deleting a removable edge have bounded size, as Lemma~\ref{lem:remove_cuts} ensures, but, in fact, they are pendant trees.
For every class $\cG_n$ and every $t\geq 1$, if $G_n$  is chosen uniformly at random from $\cG_n$ and $V_n$ is chosen uniformly at random from $V(G_n)$, then let $q(\cG_n,t)$ be the probability that $V_n$ is connected to the bulk of $G_n$ through a removable cut-edge and the corresponding pendant graph is a tree of order at most $t$. Observe that if $\cG$ is a subclass of forests, Lemma~\ref{lem:remove_cuts} implies that for every $\theta$, and under some conditions, 
there exists an $\ell$ such that $q(\cG_n,\ell)\geq 1-\theta$. Next lemma shows that the same holds for general classes of graphs.

\begin{lemma}\label{lem:trees hanging}
For every $\vartheta$, there exist a $\zeta$ and a $t$, such that if $\cG$ is a $\zeta$-tight bridge-addable class and $n$ is large enough, then 

$q(\cG_n,t)\geq 1-\vartheta$.
\end{lemma}

\begin{proof}
Given $G\in \cG_n$ and a vertex $v\in V(G)$ that is connected to the bulk of $G$ through a cut-edge $e$, we denote by $X_G(v)$ the pendant graph (containing $v$) obtained when deleting $e$ from~$G$. 
Given $G_n$ chosen uniformly at random from $\cG_n$ and $V_n$ chosen uniformly at random from $V(G_n)$, as before, we define $A_1$ as the event that $V_n$ is connected to the bulk of $G_n$ through a removable cut-edge and  $A_2$ as the event that $X_{G_n}(V_n)$ has order at most $t$. Also, let $A_3$ be the event that $X_{G_n}(V_n)$ is a tree. 
It is implicit it the definition of $A_2$ and $A_3$ that $V_n$ should be connected to the bulk of $G_n$ through a cut-edge, so in particular $X_{G_n}(V_n)$ is well-defined.
Note that:
\begin{align}\nonumber
q(\cG_n,t)&=\Pr(A_1\cap A_2\cap A_3)= \Pr(A_1\cap A_2) -\Pr(A_1\cap A_2\cap\overline{A_3}) \\
&\geq \Pr(A_1\cap A_2) -\Pr( A_2\cap\overline{A_3})\;,\label{eq:splitProba}
\end{align}
so we will proceed by bounding the last two probabilities.

Here we consider again the partition of $\cG_n$ into subclasses $\cH^{[1]}_n, \cH^{[2]}_n,\dots$ defined above. Given $\zeta_0$ (to be fixed later), there exists a $\zeta$ such that for every $\zeta$-tight class $\cG$, if $n$ is large enough, we can consider $S_n=S_n(\zeta_0)$ to be the set of indices given by Lemma~\ref{lemma:tightsubclasses}. We let $\cH:=\cH^{[i]}_n$ be the corresponding subclass of $\cG_n$, for some $i\in S_n$, and $\cF_{\cH}$ be corresponding class of forests.

By Lemma~\ref{lem:remove_cuts} with $\theta=\vartheta/3$, if $\zeta_0$ is small enough and, $n$ and $t$ are large enough, since $\cH$ is a $\zeta_0$-tight bridge-addable class of graphs with $n$ vertices, then the probability that a uniformly chosen vertex $W_n$ from a uniformly chosen forest $F_n$ in $\cF_{\cH}$ connects to the bulk of $F_n$ through a removable cut-edge and that $X_{F_n}(W_n)$ is a tree of order at most $t$, is at least $1-\vartheta/3$. If this is the case, as we argued before, this edge is also a removable cut-edge in the graph in $\cH$ that corresponds to $F_n$.
Thus, using~\eqref{eq:Hi_tight2} and provided that $\zeta_0$ is small enough with respect to $\vartheta$, we can lower bound the first probability in~\eqref{eq:splitProba} as follows
 \begin{align}\label{eq:A_1 A_2}
\Pr(A_1\cap A_2)\geq 1-\frac{\vartheta}{3}-\zeta_0\geq 1-\frac{\vartheta}{2}\;.
\end{align} 

It remains to obtain an upper bound on $\Pr(A_2\cap \overline{A_3})$.
For this we first observe that, by using Lemma~\ref{lem:remove_cuts} again with $\theta_2=\frac{\vartheta}{7t}$, and if $\zeta_0$ is small enough, and, $n$ and $\ell$ are large enough, since $\cH$ is $\zeta_0$-tight, then the probability that a uniformly chosen vertex $W_n$ from a uniformly chosen forest $F_n$ in $\cF_{\cH}$ is connected to the bulk of $F_n$ through a removable cut-edge, is at least 
$1-\frac{\vartheta}{7t}$.
Using \eqref{eq:Hi_tight2} again and provided that $\zeta_0$ is small enoughwith respect to $\vartheta$ and $t$, we obtain 
\begin{align}\label{eq:A1}
\Pr(\overline{A_1})\leq \frac{\vartheta}{7t}+\zeta_0
\leq \frac{\vartheta}{6t}\;.
\end{align} 
 We claim that
\begin{align}\label{eq:claimProba123}
\Pr(A_2\cap \overline{A_3})
\leq t \Pr(\overline{A_1})
\;.
 \end{align}
Assuming that~\eqref{eq:claimProba123} holds, together with~\eqref{eq:A_1 A_2} and with~\eqref{eq:A1}, we can now control each term in~\eqref{eq:splitProba} to obtain
$$
q(\cG_n,t)  \geq  1-\frac{\vartheta}{2} - \frac{\vartheta}{6}  \geq 1-\vartheta\;.
$$

Thus, it only remains to prove~\eqref{eq:claimProba123}. For this we observe that if $A_2\cap \overline{A_3}$ holds, then $X_{G_n}(V_n)$ contains at least one vertex $V'_n$ which is not connected to the bulk of $G_n$ through a cut-edge (since $X_{G_n}(V_n)$ is a well-defined pendant graph, but it is not a tree). Moreover since $A_2$ holds, the graph distance between $V_n$ and $V'_n$ is less than $t$.
Conversely, it is easy to see that given any vertex $v'$, there are at most $t$ vertices $v$ at distance at less than $t$ from $v'$ that are connected to the bulk of $G_n$ through a cut-edge and such that $X_{G_n}(v)$ contains $v'$.
The inequality~\eqref{eq:claimProba123} thus follows by double-counting such pairs of vertices.
\end{proof}

\subsection{Proof of Theorem~\ref{thm:main}}
\label{subsec:conclude}

We finally present the proof of our main theorem.
\begin{proof}[Proof of Theorem~\ref{thm:main}]
Let us first prove $i)$. We will first prove that for every $k$, every $\theta$ and every $U_1,\dots, U_k$, and if $\zeta$ is small enough and $n$ large enough, then for every $\zeta$-tight bridge-addable class $\cG$, we have
\begin{align}\label{eq:main2}
 \left| \frac{ \left| \cG_n^{k+1,\{U_1,\dots,U_k\}} \right|}{|\cG_n|}- e^{-1/2}\frac{e^{-\sum_{i=1}^k |U_i|}}{\Aut_u(U_1,\dots,U_k)}\right|< \theta\;.
\end{align}

As before we consider the partition of $\cG_n$ into subclasses  $\cH^{[1]}_n, \cH^{[2]}_n,\dots$. Given $\zeta_0$ (to be fixed later), there exists a $\zeta$ such that for every $\zeta$-tight class $\cG$, if $n$ is large enough, we can consider the set $S_n=S_n(\zeta_0)$ given by Lemma~\ref{lemma:tightsubclasses}.

Let $\cH:=\cH^{[i]}_n$ for $i \in S_n$ and let $\cF_{\cH}$ be the corresponding $\zeta_0$-tight class of forests. 
We can apply Theorem~\ref{thm:main_forests} for the given $k$, $\theta_k=\frac{\theta}{4}$, and the given $U_1,\dots,U_k$. If $\zeta_0$ is small enough and $n$ is large enough, and since $\cF_\cH$ is $\zeta_0$-tight,~\eqref{eq:main2} holds for $\cF_\cH$. 

It follows that
\begin{align*}
\left| \cG_n^{k+1,\{U_1,\dots,U_k\}} \right|
& =  \sum_{j\in S_n} \left| (\cH^{[j]}_n)^{k+1,\{U_1,\dots,U_k\}} \right| \pm \zeta_0|\cG_n| \nonumber\\
& =  \sum_{j\in S_n} \left| \cF_{{\cH^{[j]}_n}}^{k+1,\{U_1,\dots,U_k\}} \right| \pm \zeta_0|\cG_n| \nonumber\\
& =  \left( e^{-1/2}\frac{e^{-\sum_{i=1}^k |U_i|}}{\Aut_u(U_1,\dots,U_k)} \pm\theta_k\right)\sum_{j\in S_n} |\cF_{{\cH^{[j]}_n}}| \pm \zeta_0|\cG_n|\nonumber \\
& = \left( e^{-1/2}\frac{e^{-\sum_{i=1}^k |U_i|}}{\Aut_u(U_1,\dots,U_k)} \pm\theta_k\right)(1\pm \zeta_0)|\cG_n| \pm \zeta_0|\cG_n|\nonumber \\
&=   \left( e^{-1/2}\frac{e^{-\sum_{i=1}^k |U_i|}}{\Aut_u(U_1,\dots,U_k)} \pm\theta\right) |\cG_n| \;,
\end{align*}
provided that $\zeta_0$ is small enough with respect to $\theta$. This proves~\eqref{eq:main2}. 

To prove the first part of the theorem, let $k_*$ be large enough such that 
\begin{align}\label{eq:choice_of_ell}
e^{-1/2} \sum_{k=0}^{k_*} \sum_{\{U_1,\dots,U_k\}\in \cU_{\leq k_*}} \frac{e^{-\sum_{i=1}^k |U_i|}}{\Aut_u(U_1,\dots,U_k)} &\geq 1-\frac{\epsilon}{4}\;.
\end{align}
The existence of such a $k_*$ is, again, guaranteed by Lemma~\ref{lemma:evalGF}.

If ${\bf f}$ is an unrooted unlabeled forest composed by trees $U_1,\dots, U_k$, then 
$$
\Pr(\sma(G_n)\equiv{\bf f}) = \frac{\left| \cG_n^{k+1,\{U_1,\dots,U_k\}} \right|}{|\cG_n|}\;.
$$
We choose $\theta:=\epsilon k_*^{-k_*^2}/2$. 

Let ${\bf f_1}$ be a forest composed by at most $k_*$ trees of size at most $k_*$, then~\eqref{eq:main2} gives that $|\Pr(\sma(G_n)\equiv{\bf f_1}) -p_\infty({\bf f_1})|<\epsilon$. 

Let ${\bf f_2}$ be a forest with either more than $k_*$ trees or where at least one of the trees has size larger than $k_*$.
Since $p_\infty$ is a probability distribution, by~\eqref{eq:choice_of_ell} we have $p_\infty({\bf f_2})\leq \epsilon/4$. Since $\sum_{\bf f} \Pr(\sma(G_n)\equiv{\bf f})=1$, using again~\eqref{eq:main2} and~\eqref{eq:choice_of_ell}, we have
\begin{align*}
|\Pr(\sma(G_n)\equiv{\bf f_2}) -p_\infty({\bf f_2})|&\leq  \Pr(\sma(G_n)\equiv{\bf f_2}) +p_\infty({\bf f_2})\\
&\leq  1- \sum_{k=0}^{k_*} \sum_{\{U_1,\dots,U_k\}\in \cU_{\leq k_*}}  \frac{\left| \cG_n^{k+1,\{U_1,\dots,U_k\}} \right|}{\left|\cG_n \right|}+p_\infty({\bf f_2})\\
&\leq \epsilon/4 + \theta k_*^{k_*^2} + \epsilon/4=\epsilon
\end{align*}
This concludes the proof of $i)$.
\vspace{0.3cm}

Let us now prove the following property from which $ii)$ follows directly. 
\begin{enumerate}
\item[\emph{iii)}]~\textit{for every $\epsilon, \eta$, there exists $\zeta$ such that for every $\zeta$-tight bridge-addable class $\cG$ and every $n$ large enough, if ${\bf f}$ is a fixed unrooted unlabeled forest, we have:}
\vspace{-4mm}
$$ 
\left|\Pr\left(\sma(G_n)\equiv {\bf f};\ 
\forall T \in \mathcal{T}: \ 
\left|\frac{\alpha^{G_n}(T)}{n}-a_\infty(T)\right|<\eta
\right)
-
p_\infty({\bf f})
\right|
<\epsilon.
$$
\end{enumerate}

\noindent Similarly as before, we first prove that for every $\theta,\eta$, $k,\ell$ and $U_1,\dots,U_k$, and provided that $\zeta$ is small enough and $n$ large enough, we have
\begin{align}\label{eq:this_one} 
\frac{\sum_{\beta\in \Xi(\eta,\ell)}\left| \cG_{n,\beta}^{k+1,\{U_1,\dots,U_k\}} \right|}{|\cG_{n}^{k+1,\{U_1,\dots,U_k\} } |}&\geq  1 - \theta \;.
\end{align}

Recall the partition of $\cG_n$ into subclasses  $\cH^{[1]}_n, \cH^{[2]}_n,\dots$ . As before, let $\cH:=\cH^{[i]}_n$ for $i \in S_n$ and let $\cF_{\cH}$ be the corresponding class of forests. If we apply the second part of Theorem~\ref{thm:main_forests} with $\theta_k=\theta/4$ and $\delta=\eta/2$ to the class $\cF_\cH$, if $\zeta_0$ is small enough and $n$ large enough, then at least 
$(1-\theta/2)|\cH^{k+1,\{U_1,\dots,U_k\}}|$ graphs $G\in \cH^{k+1,\{U_1,\dots,U_k\}}$ satisfy $\alpha^{F_{G}}\in\Xi(\delta, \ell)$.

Theorem~\ref{thm:main_forests} also shows that there exists a constant $c_1>0$ such that $|\cH^{k+1,\{U_1,\dots,U_k\} } |\geq c_1 |\cH|$. By Lemma~\ref{lem:trees hanging} with $\vartheta:=c_1\min\{\theta/4,\delta\}$, if $\zeta_0$ is small enough and $n$ large enough there exists a $t$ such that with probability at least $1-\vartheta$, a random vertex  in a random graph of $\cH$ is connected via a removable cut-edge and the corresponding pending graph is a tree of order at most $t$. We can choose $t\geq \ell$. (Note that by doing so, we only increase the former probability.) 

Therefore, if $H_n$ is a random graph in $\cH^{k+1,\{U_1,\dots,U_k\} } $, with probability at least $1-\theta/2-\vartheta/c_1 >1-3\theta/4$, for every $T\in \cT_{\leq \ell}$,
$$
\frac{\alpha^{H_n}(T)}{n}=\frac{e^{-|T|}}{\Aut_r(T)}\pm\delta\pm \vartheta = \frac{e^{-|T|}}{\Aut_r(T)}\pm 2\delta\;.
$$
In other words, with probability at least $1-3\theta/4$, we have $\alpha^{H_n}\in \Xi(2\delta, \ell)=\Xi(\eta, \ell)$.

By $i)$, we have that $|\cG_{n}^{k+1,\{U_1,\dots,U_k\} } |\geq c_2 |\cG_n|$, for some constant $c_2>0$. Therefore, there are at most $\frac{\zeta_0}{c_2}|\cG_{n}^{k+1,\{U_1,\dots,U_k\}}|$ graphs in classes ${\cH^{[i]}_n}$ that are not $\zeta_0$-tight. We conclude that, provided $\zeta_0$ is small enough, the probability that a graph $G'_n$ chosen at random from $\cG_{n}^{k+1,\{U_1,\dots,U_k\}}$ satisfies
$\alpha^{G'_n}\in\Xi(\eta, \ell)$, is at least $1-3\theta/4-\zeta_0/c_2>1-\theta$. This proves~\eqref{eq:this_one}.

Let $A(k,\nu)$ the event that for every $T\in T_{\leq k}$ we have $\left|\frac{\alpha^{G_n}(T)}{n}-a_\infty(T)\right|<\nu$ (we might write $k=\infty$ where $\cT_{\leq \infty}=\cT$).

Since we have already proved $i)$, we have that for every unrooted unlabeled  forest  ${\bf f}$ with small components $U_1,\dots,U_k$, then
\begin{align}\label{eq:first}
\Pr\left(A(\infty,\eta), \sma(G_n)\equiv {\bf f} \right)& \leq \Pr(\sma(G_n)\equiv {\bf f} )\leq  p_\infty({\bf f})+\epsilon\;.
\end{align}
By Lemma~\ref{lemma:evalGF}, if $k_*$ is large enough, then
$$
\sum_{T\in \cT_{\leq k_*}} \frac{e^{-|T|}}{\Aut_r(T)}> 1-\frac{\eta}{4}\;.
$$

Let $T'\notin \cT_{\leq k_*}$ and choose  $\rho= \eta k_*^{-k_*}/4$. As before, by the properties of $k_*$ we have that $a_\infty(T')\leq \eta/4$  and, conditional on $A(k_*,\rho)$,  $\frac{\alpha^{G_n}(T')}{n}\leq \eta/4 + \rho k_*^{k_*}= \eta/2$.
This implies that, conditional on $A(k_*,\rho)$, then $A(k_*,\eta)$ implies  $A(\infty,\eta)$.

If ${\bf f}$ is an unrooted unlabeled  forest with small components $U_1,\dots,U_k$, using~\eqref{eq:this_one}, we have that for every $\theta$,
\begin{align*}
&P\left(A(\infty,\eta)\mid \sma(G_n)\equiv {\bf f} \right)\\
& \geq P\left(A(\infty,\eta)\mid \sma(G_n)\equiv {\bf f},A(k_*,\rho) \right)  \cdot P\left(A(k_*,\rho)\mid \sma(G_n)\equiv {\bf f}\right)\\
& \geq P\left(A(k_*,\eta)\mid \sma(G_n)\equiv {\bf f},A(k_*,\rho) \right)  \cdot P\left(A(k_*,\rho)\mid \sma(G_n)\equiv {\bf f}\right)\\
&=\frac{\sum_{\beta\in \Xi(\eta,k_*)} |\cG_{n,\beta}^{k+1,\{U_1,\dots,U_k\}}|}{\sum_{\beta\in \Xi(\rho,k_*)} |\cG_{n,\beta}^{k+1,\{U_1,\dots,U_k\}}|} \cdot \frac{\sum_{\beta\in \Xi(\rho,k_*)} |\cG_{n,\beta}^{k+1,\{U_1,\dots,U_k\}}|}{|\cG_{n}^{k+1,\{U_1,\dots,U_k\}}|}\\
&=\frac{\sum_{\beta\in \Xi(\eta,k_*)} |\cG_{n,\beta}^{k+1,\{U_1,\dots,U_k\}}|}{|\cG_{n}^{k+1,\{U_1,\dots,U_k\}}|}\\
&\geq1-\theta
\end{align*}
By $i)$, we may assume that $\Pr(\sma(G_n)\equiv {\bf f})\geq p_\infty({\bf f})-\epsilon/2$. Choosing $\theta:=\epsilon/2$, we conclude 
\begin{align*}
P\left(A(\infty,\eta), \sma(G_n)\equiv {\bf f} \right)
& = P\left(A(\infty,\eta)\mid \sma(G_n)\equiv {\bf f} \right) P\left(\sma(G_n)\equiv {\bf f} \right) \\
&\geq (1-\theta)(p_\infty({\bf f})-\epsilon/2)\geq p_\infty({\bf f})-\epsilon\;.
\end{align*}
Together with~\eqref{eq:first}, this proves \emph{iii)}, which directly proves \emph{ii)}.
\end{proof}
\label{pageMainbis}

\subsection{Proof of Corollary~\ref{cor:BS}}
\label{subsec:BS}

Corollary~\ref{cor:BS} is a simple consequence of our main result. Here we present a detailed proof.

\begin{proof}[Proof of Corollary~\ref{cor:BS}]

Let $G$ be a graph with $n$ vertices, let $v\in V(G)$  and $r\geq 1$. The \emph{ball} of radius~$r$ centered at $v$,  $B_{G,r}(v)$, is the graph induced in $G$ by all vertices at distance at most $r$ from $v$. 
The \emph{hull} of radius $r$, $H_{G,r}(v)$ is the union of $B_{G,r}(v)$ with all the connected components of $G\setminus B_{G,r}(v)$ that are of size smaller than~$\frac{n}{3}$, but are not components of $G$.
We view  
the hull $H_{G,r}(v)$ as a graph with a root (the vertex $v$) and a set, possibly empty,  of \emph{exit vertices} (the vertices to which component(s) of size larger than $\frac{n}{3}$ are attached). Note that the exit vertices are necessarily at distance $r$ from the root.
We extend the definition of hulls to infinite graphs, by replacing the condition ``size smaller than $\frac{n}{3}$'' by the condition ``finite size''.

For $k\geq 0$, let $\mathcal{T}_{r,k}$ be the set of (unlabeled) trees with a marked root, and $k$ marked distinct  vertices at distance $r$ from the root (exit vertices).
For $T\in \mathcal{T}_{r,k}$ and a rooted graph $(G,v)$, we write $H_{G,r}(v) \equiv T$ if the hull $H_{G,r}(v)$ is isomorphic to $T$ as an unlabeled graph, where the isomorphism preserves the root and the exit vertices (in particular this implies that $H_{G,r}(v)$ has $k$ exit vertices).
Then it is easy to see from the definition of $(F_\infty,V_\infty)$ that we have, for any $r\geq1, k\geq0$ and $T\in \cT_{r,k}$:
$$
\Pr \left( H_{F_\infty,r}(V_\infty) \equiv T\right)
=
q_\infty(T),
$$
where 
$$
q_\infty(T):=
\left\{
\begin{array}{rl}
\frac{1}{\Aut_{path}(T)}e^{-|T|} &\mbox{if } k=1
\\
0 &\mbox{if } k\neq 1,
\end{array} \right.
$$
where $\Aut_{path}(T)$ is the number of automorphisms of $T$ preserving the path from the root to the exit vertex. Moreover, for any $r\geq 1$ 
we have
\begin{align}
\label{eq:1endIsEnough}
\sum_{k\geq0}\sum_{T\in\mathcal{T}_{r,k}}  q_\infty(T) =1
.\end{align}

Let $r\geq 1$ and fix $T\in \mathcal{T}_{r,1}$, with root $u$ and exit vertex $w$. Let $T'$ be the element of $\mathcal{T}$ obtained by re-rooting the tree $T$ at $w$, and let $m$ be the number of copies of the vertex $u$ in $T'$. Then, clearly, there are \emph{at least}
$
m \alpha^{G}(T')
$
vertices $v$ in $V(G)$ such that $H_{G,r}(v) \equiv T$.
We thus have, 
$$\Pr \left(
H_{G,r}(V) \equiv T
\right) \geq \frac{m \alpha^{G}(T')}{n},$$ 
where $V$ is a uniform random vertex in $G$.
Now let $\cG$ be a tight bridge-addable graph class, and, for every $n\geq 1$, let $G_n$ be a uniform graph in $\cG_n$ and let $V_n$ be a uniform vertex in $G_n$. By averaging over graphs in $\cG_n$ and using the second part of Theorem~\ref{thm:tight} we obtain:
\begin{align}
\label{eq:limitFor1End}
\liminf_n \Pr \left(
H_{G_n,r}(V_n) \equiv T
\right)
\geq  \liminf_n \mathbf{E} \left(\frac{m \alpha^{G_n}(T')}{n}\right)
\geq m a_\infty(T') 
=q_\infty(T)
\end{align}
where for the last equality we used $ m \Aut_{path}(T) = \Aut_r(T')$.
Now since the events $H_{G_n,r}(V_n) \equiv T$ for $T\in \cup_{k\geq 0} \mathcal{T}_{r,k}$ are disjoint, we have:
$$
\sum_{k\geq 0} \sum_{T\in \mathcal{T}_{r,k}}
\Pr \left(
H_{G_n,r}(V_n) \equiv T
\right) 
\leq 1.
$$
From~\eqref{eq:1endIsEnough} and~\eqref{eq:limitFor1End} we thus get that, for any $r,k$ and $T\in \mathcal{T}_{r,k}$ we have: 
\begin{align}\label{eq:BShulls}
\lim_n \Pr \left(
H_{G_n,r}(V_n) \equiv T
\right)
=q_\infty(T).
\end{align}

The last equation implies that, for any rooted graph $B_0$ of radius $r$ (where the radius is the greatest distance from a vertex to the root), we have:
\begin{align}\label{eq:BSballs}
\lim_n\Pr \left(
B_{G_n,r}(V_n) \equiv B_0
\right)
=\Pr \left(
B_{F_\infty,r}(V_\infty) \equiv B_0
\right).
\end{align}
To see this, note that for every rooted graph $B$ we have:
\begin{align*}
\Pr \left(
B_{F_\infty,r}(V_\infty)\equiv B\right)=\sum_{k \geq 0} \sum_{T\in\mathcal{T}_{r,k}\atop T\triangleright B}\Pr \left(
H_{F_\infty,r}(V_\infty)\equiv T\right),
\end{align*}
where $T\triangleright B$ means that $B_{T,r}(v)\equiv B$, where $v$ is the root of $T$.

It follows from this equality that for any $B$, any $r\geq 1$ and any $\epsilon$,
we can choose a finite subset $\mathcal{T}'\subset \cup_{k\geq 0} \cT_{r,k}$ such that  $\sum_{T\in\mathcal{T}', T\triangleright B}\Pr \left(
H_{F_\infty,r}(V_\infty)\equiv T\right) \geq \Pr \left(
B_{F_\infty,r}(V_\infty)\equiv B\right)-\epsilon$.
Using~\eqref{eq:BShulls}, it follows that 
\begin{align*}
\liminf_n \Pr \left(
B_{G_n,r}(V_n) \equiv B
\right) 
&\geq \liminf_n
\sum_{T\in\mathcal{T}', T\triangleright B}\Pr \left(
H_{G_n,r}(V_n)\equiv T\right)\\
&
\geq 
\sum_{T\in\mathcal{T}', T\triangleright B}\Pr \left(
H_{F_\infty,r}(V_\infty)\equiv T\right)\\
&\geq 
\Pr \left(
B_{F_\infty,r}(V_\infty)\equiv B\right) -\epsilon.
\end{align*}
Since this is true for any $\epsilon>0$ we thus have proved 
\begin{align}\label{eq:BSliminf}
\liminf_n\Pr \left(
B_{G_n,r}(V_n) \equiv B
\right) 
\geq
\Pr \left(
B_{F_\infty,r}(V_\infty)\equiv B\right).
\end{align}
It follows that
$$
1\geq \liminf_n \sum_{B}\Pr \left(
B_{G_n,r}(V_n) \equiv B
\right)\geq  \sum_{B}\Pr \left(B_{F_\infty,r}(V_\infty)\equiv B\right)=1,
$$
where the sums are taken over all rooted graphs $B$ of radius $r$, and using~\eqref{eq:BSliminf}, Equation~\eqref{eq:BSballs} holds for every $B_0$. This concludes the proof.
\end{proof}

\bibliographystyle{alpha}
\bibliography{biblio}

\newpage
\appendix
\begin{center}
 -------------- APPENDIX --------------
\end{center}

\section{Generating functions of trees and forests}
\label{app:gf}

The purpose of this appendix is to recall some basic properties of uniform random forests and their generating functions, in particular proving Theorem~\ref{thmalpha:randomForests}  of the introduction and Lemma~\ref{lemma:evalGF} of Section~\ref{subsec:evalGF}. The material of Section~\ref{subapp:components} is a standard exercise 
in analytic combinatorics (for background, see the book~\cite{flajolet}) but we are not aware of a reference containing precisely these results, so we prefer to recall them here. Section~\ref{subapp:example} gives details about the example from Remark~\ref{rem:example} in the introduction.

\subsection{Small components and pendant trees in uniform forests}
\label{subapp:components}

For $n\geq 1$ we let $u_n, f_n$ be respectively the number of labeled trees and labeled forests on $n$ vertices. Thus $t_n:=nu_n$ is the number of labeled rooted trees on $n$ vertices.

We let 
$T(z)=\sum_{n\geq 1} \frac{t_n}{n!}z^n$,
$U(z)=\sum_{n\geq 1} \frac{u_n}{n!}z^n$,
$F(z)=\sum_{n\geq 0} \frac{f_n}{n!}z^n$,
be the exponential generating functions of labeled: rooted trees, unrooted trees, and forests, respectively. Note that the sum defining $F(z)$ starts at $n=0$ since we allow the empty forest in this discussion, and we set $f_0:=1$ accordingly.
Since each rooted tree can be uniquely transformed, by removing the root, into a set made of smaller rooted trees, we have using the classical dictionnary of symbolic combinatorics~\cite{flajolet}:
\begin{align}\label{appeq:T}
T(z)=z \exp T(z).
\end{align}
Note that the closed formula $t_n=n^{n-1}$ follows directly from the last equation and the Lagrange inversion formula.

In order to get an equation for the series $U(z)$, we note that any tree  
$\mathbf{t}$ satisfies the equation 1=\#vertices$(\mathbf{t})$-\#edges$(\mathbf{t})$. 
By summing this equation over all trees of given size, and observing that the generating function of trees rooted at an edge is given by $\frac{1}{2}T(z)^2$, we obtain the \emph{disymmetry equation} (see~\cite{quebecois}):
\begin{align}\label{appeq:U}
U(z)=T(z)-\frac{1}{2}T(z)^2.
\end{align}
Finally, since any forest can be viewed as a set of \emph{unrooted} trees, we have:
\begin{align}\label{appeq:F}
F(z) = \exp U(z).
\end{align}

\medskip

The above equations  express implicitly the functions $T(z),U(z),F(z)$ in terms of the variable $z$. In order to transform these implicit equations into asymptotic estimates, we will use so called \emph{transfer} theorems. These results (see \cite[Chap VI]{flajolet}) enable one to deduce precise estimates on the coefficients of a generating function given an expansion at the main singularity of this function on its disk of convergence.

First, it follows by \eqref{appeq:T} and the implicit function theorem (see also the general theorems in \cite[Chap VII]{flajolet}) that the function $T(z)$ has radius of convergence $e^{-1}$, has a unique dominant singularity at $z=e^{-1}$, and has an expansion of the form:
\begin{align}\label{appeq:Tdev}
T(z) = 1 + c \sqrt{1-ze} + c'(1-ze)+ O( (1-ze)^{3/2})
\end{align}
for some $c,c'\neq 0$, uniformly in a neighbourhood of $z=e^{-1}$ slit along the half-line $[e^{-1},\infty)$. Of course we could compute the value of $c$ and $c'$ by implicit differentiation, but this will not be needed.
It follows from~\eqref{appeq:U} that the series $U(z)$ has, uniformly in the same domain, an expansion of the form:
\begin{align}\label{appeq:Udev}
U(z) = \frac{1}{2} + c''(1-ze) + cc' (1-ze)^{3/2} + O((1-ze)^2),
\end{align}
and from~\eqref{appeq:F} we obtain the expansion of the function $F(z)$:
\begin{align}\label{appeq:Fdev}
F(z) = e^{1/2} + c'''(1-ze) + e^{1/2} cc' (1-ze)^{3/2} + O((1-ze)^2).
\end{align}
Note that Equations~\eqref{appeq:Tdev} and \eqref{appeq:Fdev} 
 imply Lemma~\ref{lemma:evalGF} stated in Section~\ref{subsec:evalGF}.

\medskip
We are now ready to reprove the result of R\'enyi about the probability that a forest is connected~\cite{Renyi}. Applying a standard transfer theorem for analytic functions~\cite{flajolet}, \eqref{appeq:Udev} implies that the number of (unrooted) trees of size $n$ satisfies, when $n$ goes to infinity:
$$
\frac{u_n}{n!}\sim \frac{cc'}{\Gamma(-3/2)} n^{-5/2} e^n. 
$$
Similarly, for $f_n$, \eqref{appeq:Fdev} implies:
\begin{align}\label{appeq:fn}
\frac{f_n}{n!}\sim \frac{cc'e^{1/2}}{\Gamma(-3/2)} n^{-5/2} e^n. 
\end{align}
from which we get Renyi's result:
$$
\frac{u_n}{f_n} \rightarrow e^{-1/2}.
$$
\medskip
More generally, fix a forest $U\in \cF$, \emph{possibly empty}, 
and let $F^U(z)$ be the generating function of all forests $\mathbf{f}$ such that $\mathop{Small}(\mathbf{f})\equiv U$. Since there are $\frac{|U|!}{\mathrm{Aut}(U)}$ labeled forests isomorphic to $U$, we have:
$$
F^U(z) = \frac{z^{|U|}}{\mathrm{Aut}(U)} U(z) + P_U(z),
$$
where $P_U$ is a \emph{polynomial} taking into account the cases of small forests containing $U$ as a subforest but in which the largest component is among the components of $U$.
The previous equation is also valid when $U$ is empty (with $\frac{z^{|U|}}{\mathrm{Aut}(U)}:=1$).

From our previous estimates, the function $F^U$ has an asymptotic expansion, uniformly in a neighbourhood of $e^{-1}$ slit along $[e^{-1},\infty)$ of the form:
\begin{align}\label{appeq:FUdev}
F^U(z) =c_4 + c_5(1-ze) + \frac{e^{-|U|}}{\mathrm{Aut}(U)} cc' (1-ze)^{3/2} + O((1-ze)^2).
\end{align}
(note that the unknown polynomial $P_U$ can contribute to the constant or linear term, but not the singular term of exponent $3/2$, which fortunately is the only one we are interested in for the analysis).
By standard transfer theorems, we get that when $n$ goes to infinity the coefficient of $z^n$ in this function is equivalent to:
$$
\frac{f^U_n}{n!}\sim \frac{cc'e^{-|U|}}{\mathrm{Aut}(U)\Gamma(-3/2)} n^{-5/2} e^n. 
$$

From \eqref{appeq:Fdev}, we thus obtain:
$$
\frac{f^U_n}{f_n} \longrightarrow \frac{e^{-\frac{1}{2}-|U|}}{\mathrm{Aut}(U)} = p_\infty(U), 
$$
which is precisely the first part of Theorem~B.
In order to see (as is also claimed in Theorem~B) that $p_\infty$ is a probability measure, we note that the sum over all (possibly empty) forests $U\in \cF$ given by:
$$
\sum_{U\in \cF} p_\infty(U) = e^{-1/2} \sum_{U\in \cF}\frac{e^{-|U|}}{\mathrm{Aut}(U)}
$$
is equal to 
$e^{-1/2}F(e^{-1})$
by definition of $F(z)$ (recall that $F(z)$ contains the contribution of the empty forest).
 Therefore the fact that $p_\infty$ is a probability measure is equivalent to the equation $e^{-1/2}F(e^{-1}) =1 $, which is true from~\eqref{appeq:Fdev}.

\medskip
To prove the second part of Theorem~B, we proceed once again by analysis of generating functions, together with the second moment method.
Fix a rooted (unlabeled) tree $T$. Since $T$ has $\frac{|T|!}{\mathrm{Aut}_r(T)}$ different labelings, the generating function of all labeled forests carrying a marked pendant copy of $T$ is equal to:
\begin{align}
\Big(\frac{z^{|T|}}{\mathrm{Aut}_r(T)} T(z) +Q_T(z) \Big)F(z) 
\end{align}
where the factor $T(z)$ accounts for the tree attached to the other end of the edge from which the copy of $T$ is pending, where $F(z)$ accounts for the other components, and where as before $Q_T(z)$ is a polynomial taking into account small number effects.
The expansion of this quantity when $z$ approaches $e^{-1}$ is given, using previous expansions, by:
$$
c_6 + \frac{c \cdot e^{-|T|+1/2}}{\mathrm{Aut}_r(T)} \sqrt{1-ze} + O(1-ze).
$$
By standard transfer theorems, when $n$ goes to infinity the coefficient of $z^n$ in this function, which is the number of forests of size $n$ with a \emph{marked} pendant copy of $T$,  is equivalent to:
$$
\frac{\Gamma(-1/2)}{\Gamma(-3/2)}   \frac{c \cdot e^{-|T|+1/2}}{\mathrm{Aut}_r(T)} n^{-3/2} e^n.
$$
Dividing by $f_n/n!$ and using~\eqref{appeq:fn} we get  that the \emph{average} number $\chi_1(n)$ of pendant copies of $T$ in a random forest of size $n$ is asymptotic to:
$$
\chi_1(n)\sim\frac{\Gamma(-1/2)}{c'\Gamma(-3/2)} 
\frac{e^{-|T|}}{\mathrm{Aut}_r(T)} n=
\frac{e^{-|T|}}{\mathrm{Aut}_r(T)} n,
$$
where we used that $\frac{\Gamma(-1/2)}{c'\Gamma(-3/2)}=1$ (which can be seen either by computing explicitly the expansion~\eqref{appeq:Tdev} or whithout computation by noting that $t_n= n \cdot u_n$ and applying a transfer theorem to estimate $t_n$ and $u_n$ from \eqref{appeq:Tdev} and \eqref{appeq:Udev}).

\medskip
We have thus proved that the \emph{average} number of pendant copies of $T$ in a random forest of size $n$ is equivalent to $\frac{e^{-|T|}}{\mathrm{Aut}_r(T)} n$. To prove that  this is in fact the typical value with high probability we will compute the second moment. 
 The generating function of forests with \emph{two} marked copies of the tree $T$ is of the form:
\begin{align}
\Big(\Big(\frac{z^{|T|}}{\mathrm{Aut}_r(T)}\Big)^2 \left(\frac{zd}{dz}\right) T(z) 
+R_T(z) T(z)
+S_T(z) \Big)F(z) 
\end{align}
where the first term is interpreted as a doubly rooted tree (hence $(\frac{zd}{dz}) T(z)$) to which two non-overlapping copies of $T$ are attached, where the second term takes into account overlapping copies (here we  do not need to know $R_T(z)$ explicitly, but we know it is a polynomial) 
and the polynomial $S_T(z)$ takes into account small cases for which there is ambiguity in the choice of the largest component; as before the factor $F(z)$ takes into account the forest formed by other connected components.
The expansion at $z=e^{-1}$ of this function follows from differentiating~\eqref{appeq:Tdev} 
and is given by:
$$
\frac{-c}{2} \Big(\frac{e^{-|T|}}{\mathrm{Aut}_r(T)}\Big)^2 e^{1/2}(1-ze)^{-1/2}
+O(1)\;.
$$
Using a  transfer theorem, the coefficient of $z^n$ in this expression is asymptotic to:
$$
\frac{-c}{2\Gamma(-1/2)}\Big(\frac{e^{-|T|}}{\mathrm{Aut}_r(T)}\Big)^2 e^{1/2} n^{-1/2}e^n.
$$
Dividing by $f_n/n!$ and using~\eqref{appeq:fn}, we see that the average squared number $\chi_2(n)$ of pendant copies of $T$ in a random forest of size $n$ is asymptotic to:
$$
\chi_2(n)\sim\frac{\Gamma(-3/2)}{2c'\Gamma(-1/2)} 
\Big(\frac{e^{-|T|}}{\mathrm{Aut}_r(T)}\Big)^2 n^2
=
\Big(\frac{e^{-|T|}}{\mathrm{Aut}_r(T)}\Big)^2 n^2,
$$
where one may choose either previously mentioned method to compute the explicit value of $c'$.
We have finally obtained that 
$$\chi_2(n) \sim (\chi_1(n))^2,$$ so by the second moment method (Chebyshev's inequality) this completes the proof of the second part of Theorem~B.

\subsection{More details on the example given in Remark~\ref{rem:example}}
\label{subapp:example}

Let $\tilde{\mathcal{F}_n}$ be the class of graphs defined in Remark~\ref{rem:example}, and write $k_n:=\lceil n^{2/3}\rceil$. In this section we prove that $\tilde{\mathcal{F}}_n$ is tight.

For $i\geq k\geq 1$ let $\tilde a_{i,k}$ be the number of connected graphs on $[1..i]$ that induce a clique on $[1..k]$, and such that contracting this clique gives a tree. Thus   the number of \emph{connected} graphs in our class $\tilde{\mathcal{F}_n}$ is, by definition, equal to 
$
\tilde a_{n,k_n}.
$
Note that $\tilde a_{i,k}$ equals to the number of rooted forests on $[1..i]$ with $k$ components rooted at $1,2,\dots k$. Thus ${i \choose k} a_{i,k}$ is the number of rooted forests on $[1..i]$ with $k$ components and no condition on the location of the roots, which is classically equal to ${i-1 \choose i-k} i^{i-k}$. We thus get:
$$
\tilde a_{i,k}  =  k i^{i-k-1}. 
$$

Observe that:
\begin{align}\label{appeq:ratio}
\frac{\tilde a_{i,k}}{\tilde a_{i+1,k}} 
=  \frac{i^{i-k-1}}{(i+1)^{i-k}}
= \frac{1}{i} \left(1-\frac{1}{i+1}\right)^{i-k}\;. 
\end{align}

The number $g_n$ of all elements in the class $\tilde{\mathcal{F}}_n$ is given by:
\begin{align}\label{eq:convol}
\frac{g_n}{(n-k_n)!}
 = 
\sum_{i+j=n\atop j\geq 0, i\geq k_n} 
\frac{\tilde a_{i,k_n}}{(i-k_n)!}
\times 
\frac{f_{j}}{j!}
\end{align}
where $f_j$ counts unrooted forests, as before. In this sum, $i$ is interpreted as the number of vertices in the connected component containing the clique, and we have distributed the labeling binomial ${n-k_n\choose j}$ among factors.

From \eqref{appeq:Fdev}, given $\epsilon$ we can choose $\delta$ small enough and $j_0$ large enough such that $\sum_{j\leq j_0} f_j z^j \geq e^{1/2}(1-\epsilon)$ for any $z\geq e^{-1}-\delta$.
Also, given $\delta$ and $j_0$, for $n$ large enough, we have from~\eqref{appeq:ratio} that for any $i$ larger than $n-j_0$:
$$
\frac{\tilde a_{i,k_n}/(i-k_n)!}{\tilde a_{i+1,k_n}/(i+1-k_n)!} \geq e^{-1}-\delta\;.
$$
 
We can now lower bound the sum~\eqref{eq:convol} by keeping the contribution of relatively small values of~$j$. More precisely, for $n$ large enough, we have:
\begin{align*}
\eqref{eq:convol}& 
\geq \sum_{j\leq j_0} 
\frac{\tilde a_{n-j,k_n}}{(n-j-k_n)!} \frac{f_{j}}{j!}
\\
&\geq  \frac{\tilde a_{n,k_n}}{(n-k_n)!} \sum_{j\leq j_0} (e^{-1}-\delta)^j \frac{f_j}{j!}\\
&\geq  \frac{\tilde a_{n,k_n}}{(n-k_n)!} e^{1/2} (1-\epsilon)\;.
\end{align*}
Given $\zeta$, consider $\epsilon=\zeta/2$. If $n$ is large enough and $\tilde{F}_n$ is a uniform random graph in $\tilde{\mathcal{F}}_n$, we thus have
$$
\Pr(\tilde{F}_n \text{ is connected})=\frac{\tilde a_{n,k_n}}{g_n} \leq e^{-1/2} (1-\epsilon)^{-1} \leq (1+\zeta) e^{-1/2}.
$$ 
Since this is true for every $\zeta$, the class $\tilde{\mathcal{F}}$ is tight.

\end{document}